\documentclass[11pt,twoside]{article}

\usepackage[margin=1.3in]{geometry}

\usepackage{amssymb,amsmath}
\usepackage{amsthm}
\usepackage{hyperref}
\usepackage{mathrsfs}
\usepackage{textcomp}
\hypersetup{
    colorlinks,%
    citecolor=black,%
    filecolor=black,%
    linkcolor=black,%
    urlcolor=black
}

\usepackage{verbatim}
\usepackage[titletoc,toc,title]{appendix}

\usepackage{sectsty}

\makeatletter

\newdimen\bibspace
\renewenvironment{thebibliography}[1]{%
 \section*{\refname 
       \@mkboth{\MakeUppercase\refname}{\MakeUppercase\refname}}%
     \list{\@biblabel{\@arabic\c@enumiv}}%
          {\settowidth\labelwidth{\@biblabel{#1}}%
           \leftmargin\labelwidth
           \advance\leftmargin\labelsep
           \itemsep\bibspace
           \parsep\z@skip     %
           \@openbib@code
           \usecounter{enumiv}%
           \let\p@enumiv\@empty
           \renewcommand\theenumiv{\@arabic\c@enumiv}}%
     \sloppy\clubpenalty4000\widowpenalty4000%
     \sfcode`\.\@m}
    {\def\@noitemerr
      {\@latex@warning{Empty `thebibliography' environment}}%
     \endlist}

\makeatother

\makeatletter

\newtheorem{thm}{Theorem}[section]
\newtheorem{lem}[thm]{Lemma}
\newtheorem{prop}[thm]{Proposition}

\newtheorem{rem}[thm]{Remark}

\newtheorem*{rem*}{Remark}
\usepackage{fancyhdr}
\def\Xint#1{\mathchoice
  {\XXint\displaystyle\textstyle{#1}}%
  {\XXint\textstyle\scriptstyle{#1}}%
  {\XXint\scriptstyle\scriptscriptstyle{#1}}%
  {\XXint\scriptscriptstyle\scriptscriptstyle{#1}}%
  \!\int}
\def\XXint#1#2#3{{\setbox0=\hbox{$#1{#2#3}{\int}$}
  \vcenter{\hbox{$#2#3$}}\kern-.5\wd0}}

\def\dashint{\Xint-}

\newcommand{\al}{\alpha}                \newcommand{\lda}{\lambda}
\newcommand{\om}{\Omega}                \newcommand{\pa}{\partial}
\newcommand{\va}{\varepsilon}           \newcommand{\ud}{\mathrm{d}}
\newcommand{\be}{\begin{equation}}      \newcommand{\ee}{\end{equation}}
\newcommand{\w}{\omega}                 
\newcommand{\Lda}{\Lambda}              \newcommand{\B}{\mathcal{B}}
\newcommand{\R}{\mathbb{R}}

\begin{document}

\title{\textbf{Asymptotic symmetry and local behavior of solutions of higher order conformally invariant equations with isolated singularities}
\bigskip}

\author{\medskip  Tianling Jin\footnote{T. Jin was partially supported by Hong Kong RGC grants ECS 26300716 and GRF 16302217.}\quad and \quad
Jingang Xiong\footnote{J. Xiong was partially supported by NSFC 11501034, a key project of NSFC 11631002 and NSFC 11571019.}}

\date{\today}

\fancyhead{}
\fancyhead[CO]{\textsc{Conformally invariant equations with isolated singularities}}
\fancyhead[CE]{\textsc{Tianling Jin and Jingang Xiong}}

\fancyfoot{}
\fancyfoot[CO, CE]{\thepage}

\renewcommand{\headrule}{}

\maketitle

\begin{abstract} We prove sharp blow up rates of solutions of higher order conformally invariant equations in a bounded domain with an isolated singularity, and show the asymptotic radial symmetry of the solutions near the singularity. This is an extension of the celebrated theorem of Caffarelli-Gidas-Spruck for the second order Yamabe equation with isolated singularities to higher order equations.  Our approach uses blow up analysis for local integral equations, and is unified for all critical elliptic equations of order smaller than the dimension. We also prove the existence of Fowler solutions to the global equations, and establish a sup$\,*\,$inf type Harnack inequality of Schoen for integral equations.
\end{abstract}


\section{Introduction}

In the classical paper  \cite{CGS}, Caffarelli, Gidas and Spruck proved asymptotic symmetry of positive smooth solutions to the critical semilinear elliptic equation with an isolated singularity
\be\label{eq:f-2}
-\Delta u=n(n-2)u^{\frac{n+2}{n-2}} \quad \mbox{in }B_R\setminus \{0\},
\ee
where $n\ge 3$, $B_R\subset \R^n$ is the open ball centered at the origin $\{0\}$ with radius $R>0$, and $\Delta$ is the Laplacian operator. More precisely, they proved
\begin{itemize}
\item[(i).] If $R=\infty$, i.e., $B_R=\R^n$, and the singularity $\{0\}$ is removable, then
\[
u(x)=\left(\frac{\lda}{1+\lambda^2 |x-x_0|^2}\right)^{\frac{n-2}{2}} \quad \mbox{for some }  \lambda>0, ~x_0\in \R^n.
\]

\item[(ii).] If $R=\infty$, i.e., $B_R=\R^n$, and $u $ is singular at the origin, then $u$ is radially symmetric, and all solutions are classified by ODE analysis, which are usually called the Fowler solutions or Delaunay type solutions.

\item[(iii).] If $R<\infty$ and $u$ is singular at the origin, then
\[
u(x)=\bar u(|x|)(1+O(|x|)) \quad \mbox{as }x\to 0
\]
with $
\bar u(|x|)= \dashint_{\partial B_{|x|}(0)} u(|x|\theta)\,\ud \theta$ being  the spherical average of $u$ over the sphere $\partial B_{|x|}(0)$. Furthermore, there
exists a Fowler solution $u^*$ such that
\begin{equation}\label{caf-r-1}
u(x)=u^*(|x|)(1+O(|x|^{\al}))\quad \mbox{as }x\to 0
\end{equation}
for some $\al>0$.
\end{itemize}

The motivation and importance of studying the solutions of \eqref{eq:f-2} and characterizing the singular set of $u$ were highlighted in the classical work of Schoen and Yau \cite{Schoen, SYau} on conformally flat manifolds and the Yamabe problem. Subsequent to \cite{CGS}, there have been many papers related to this result, including Chen-Lin \cite{ChLin, ChenLin2}, C. Li \cite{Lc}, Korevaar-Mazzeo-Pacard-Schoen \cite{KMPS}, Marques \cite{marques-1}, among others. In particular, \cite{KMPS} gives a proof of \eqref{caf-r-1} and provides an expansion of $u$ after the order of $u^*$. The main theorems of \cite{CGS} and \cite{KMPS}  were completely extended  to solutions of second order fully nonlinear Yamabe equations by Li-Li \cite{LL}, Y.Y. Li \cite{Li06}, Chang-Han-Yang \cite{CHY}  and Han-Li-Teixeira \cite{HLT}.

There are also works on studying similar problems for singular solutions of higher order conformally invariant equations with isolated singularities. The first part of the results in \cite{CGS}, i.e., the $R=\infty$ case with $\{0\}$ being removable, was extended to  fourth order equations by Lin \cite{Lin}, to other higher order equations by Wei-Xu \cite{WX}, and to conformally invariant integral equations by Chen-Li-Ou \cite{CLO} and Y.Y. Li \cite{Li04}. 

The second part of the results in \cite{CGS}, i.e., the $R=\infty$ case with $\{0\}$ being not removable, has also been extended to fourth order equations. It was known from  Lin \cite{Lin} that all the singular solutions of $\Delta^2u=u^{\frac{n+4}{n-4}}$ in $\R^n\setminus\{0\}$, $n\ge 5$, are radially symmetric. Recently, Frank-K\"onig \cite{FK} classified all the singular solutions, i.e., the Fowler solutions, by ODE analysis. See also some earlier results (including the existence of Fowler solutions, among others) by Guo-Huang-Wang-Wei \cite{GHWW}.

In this paper, we are interested in the third part of the results in \cite{CGS}, i.e., the $R<\infty$ case with $\{0\}$ being not removable, for higher order equations.

The first main result of the paper is as follows.

\begin{thm}\label{thm:higher-order} Suppose  that $1\le m<n/2$ and $m$ is an integer.   Let $u\in C^{2m}(B_1\setminus \{0\})$ be a positive solution of
\begin{align} 
(-\Delta)^m u&= u^{\frac{n+2m}{n-2m}}\quad \mbox{in }B_1\setminus \{0\}.\label{eq:h-1}
\end{align}
Suppose 
\begin{align} 
(-\Delta)^k u &\ge 0 \quad \mbox{in }B_1\setminus \{0\} , \quad  k=1,\dots, m-1. \label{eq:h-2}
\end{align}
Then either $0$ is a removable singularity of $u$, or there exists $C>0$ such that
\be\label{eq:blowuprate0}
\frac{1}{C}|x|^{-\frac{n-2m}{2}} \le u(x) \le C |x|^{-\frac{n-2m}{2}},
\ee
and
\be\label{eq:asymptoticsymmetry0}
u(x)= \bar u(|x|)(1+O(|x|)) \quad \mbox{as }x\to 0.
\ee
\end{thm}

\begin{rem}\label{rem:signconditions}
If a positive solution $u$ of \eqref{eq:h-1}  satisfies  either the lower bound in \eqref{eq:blowuprate0}, or the upper bound in \eqref{eq:blowuprate0} and that $0$ is non-removable, then the sign assumptions \eqref{eq:h-2} hold in a small punctured ball $B_\tau\setminus\{0\}$ with some $\tau>0$. See Proposition \ref{prop:sufficientcondition1} and Proposition \ref{prop:sufficientcondition2}.
 \end{rem}
 
Singular local solutions of $\Delta^2 u=u^p$ has been studied by Soranzo \cite{S}  with $p<\frac{n}{n-4}$ which extends earlier results of Brezis-Lions \cite{BL} and Lions \cite{Lions} for $m=1$, and has been studied by Yang \cite{YangHui} with $\frac{n}{n-4}<p<\frac{n+4}{n-4}$ which extends several results of  Gidas-Spruck \cite{GS} for $m=1$.

The equation \eqref{eq:h-1} is conformally invariant in the sense that if $u$ is a solution then its Kelvin transform
\begin{equation}\label{eq:kelvintransform}
u_{x,\lda }(\xi)= \left(\frac{\lda}{|\xi-x|}\right)^{n-2m} u\left( x+\frac{\lda^2(\xi-x)}{|\xi-x|^2}\right)
\end{equation}
is also a solution in the corresponding region. In the classification of entire solutions of \eqref{eq:h-1} in $\R^n$ by Lin \cite{Lin} and Wei-Xu \cite{WX}, one important step is to show that every entire positive solution of \eqref{eq:h-1} satisfies the Laplacian sign conditions \eqref{eq:h-2}. This implies that the sign conditions \eqref{eq:h-2} are kept under the Kelvin transforms \eqref{eq:kelvintransform} for entire solutions. These sign conditions will ensure the  maximum principle and are essential for applying the moving plane method. See also Gursky-Malchoidi \cite{GM} and Hang-Yang \cite{HY} for the recent progresses on the fourth order Q-curvature problem on Riemannian manifolds, where the positivity of the scalar curvature and the positivity of the Yamabe invariant are assumed, respectively. However, in our local situation \eqref{eq:h-1}, the sign conditions \eqref{eq:h-2} may change when performing the Kelvin transforms \eqref{eq:kelvintransform}.  We also note when $m=2$, the positivity of the scalar curvature of the metric $u^{\frac{4}{n-4}}\delta_{ij}$, which implies $-\Delta u>0$, is not invariant under the Kelvin transform \eqref{eq:kelvintransform}, either.

 We overcome this difficulty by rewriting the differential equations \eqref{eq:h-1} into the integral equation \eqref{eq:main-3} below, and we will work in the integral equation setting. This is inspired by a unified approach to the Nirenberg problem and its generalizations studied by the authors with Y.Y. Li in \cite{JLX}.

Assume that the dimension $n\ge 1$  and $0<\sigma<n/2$ is a real number. We consider the integral equation involving the Riesz potential
\be\label{eq:main-3}
u(x)=\int_{B_1}\frac{u(y)^{\frac{n+2\sigma}{n-2\sigma}}}{|x-y|^{n-2\sigma}}\,\ud y+h(x), \quad  x\in B_1\setminus\{0\}, \quad u>0,
\ee
where $u\in L^{\frac{n+2\sigma}{n-2\sigma}}(B_1)\cap C(B_1\setminus\{0\})$ and $h\in C^1(B_1)$ is a positive function. Under the assumptions of Theorem \ref{thm:higher-order}, one can show $u \in L^{\frac{n+2\sigma}{n-2\sigma}}_{loc}(B_1)$ and  can rewrite the equation \eqref{eq:h-1} as \eqref{eq:main-3} (after some scaling, see details in Section \ref{sec:poly}).

Our second main result is for the solutions of \eqref{eq:main-3}. Note that whenever we say $h$ is a solution of $(-\Delta)^\sigma h=0$ in some open set $\om$ for any non-integer $\sigma$, we assume that $h$ is defined in $\R^n$, is smooth in $\om$, and
\begin{equation}\label{eq:assumptionsh}
h\in W^{2m,1}_{loc}(\R^n) \mbox{ satisfying}\ \int_{\R^n}\frac{|\nabla^k h(y)|}{1+|y|^{n+2\sigma-k}}\,\ud y<\infty\ \mbox{for all }k=0,1,\cdots,2m,
\end{equation}
where $m=[\sigma]$ is the integer part of $\sigma$. 

\begin{thm}\label{thm:A} Suppose $n\ge 1$, $0<\sigma<n/2$, $u\in C(B_1\setminus \{0\}) \cap L^{\frac{n+2\sigma}{n-2\sigma}}(B_1)$ is a positive solution of \eqref{eq:main-3}, and $h\in C^1(B_1)$ is a positive function.
Then
\[
u(x)\le C_1 |x|^{-\frac{n-2\sigma}{2}}, \quad x\in B_{1/2},
\]
for some constant $C_1>0$, and
\[
u(x)= \bar u(|x|)(1+O(|x|)), \quad \mbox{as }x\to 0.
\]
Suppose in addition that $h$ is smooth in $B_1$,  $h$ is a positive function in $\R^n$ satisfying \eqref{eq:assumptionsh} if $\sigma$ is not an integer, and $h$ satisfies  $(-\Delta)^\sigma h=0$ in $B_1$.  Then either $0$ is a removable singular point of $u$ and $u\in C^\infty(B_1)$, or
\[
 u(x)\ge \frac{1}{C_2} |x|^{-\frac{n-2\sigma}{2}}
\]
for  some positive constant $C_2>0$. 
\end{thm}

Theorem \ref{thm:A} recovers Theorem \ref{thm:higher-order}. It also recovers the previous work of
Caffarelli-Jin-Sire-Xiong \cite{CJSX} for the fractional Yamabe equation ($0<\sigma<1$). Fractional Yamabe equations with a higher dimensional singular set have been studied by the authors with de Queiroz and Sire in \cite{JQSX}, and by Ao-Chan-DelaTorre-Fontelos-Gonz\'alez-Wei \cite{A+18} which extends the Mazzeo-Pacard \cite{MazzeoPard} program to a nonlocal setting.

Some blow up analysis arguments developed in Jin-Li-Xiong \cite{JLX} will be used in proving Theorem \ref{thm:A}. A difference from the nonlinear integral equations studied in \cite{CLO,Li04,JLX} is that our integral equation is locally defined, and we need to establish several delicate error estimates during the blow up. 
 When $\sigma$ is not an integer, the lower bound is much subtler. It relies on the localization formula for fractional Laplacian operators developed by Caffarelli-Silvestre \cite{CS2}, Chang-Gonz\'alez \cite{CG}, Yang \cite{Y}, Case-Chang \cite{CC}, as well as some technical analysis of Poisson integrals.

Note that if $u$ is a solution of \eqref{eq:main-3} as in Theorem \ref{thm:A} and $\sigma$ is not an integer, then one can verify that the function $\tilde h:=\int_{B_1\setminus B_{1/2}}\frac{u(y)^{\frac{n+2\sigma}{n-2\sigma}}}{|x-y|^{n-2\sigma}}\,\ud y$ is smooth in $B_{1/2}$, is positive in $\R^n$, satisfies the condition \eqref{eq:assumptionsh} and  $(-\Delta)^\sigma \tilde h=0$ in $B_{1/2}$. Thus, one always can consider the equation \eqref{eq:main-3} in a smaller ball with the same assumptions.

We are also interested in the global singular solutions of
\be\label{eq:main-6}
u(x)= \int_{\R^n}\frac{u(y)^{\frac{n+2\sigma}{n-2\sigma}}}{|x-y|^{n-2\sigma}}\,\ud y \quad \mbox{for }x\in \R^n\setminus \{0\}.
\ee
We know that $c|x|^{-\frac{n-2\sigma}{2}}$ is a solution of \eqref{eq:main-6} for a positive constant $c$ depending only on $n$ and $\sigma$. It has been proved by Chen-Li-Ou \cite{CLO1} that every smooth positive solution $u$ of \eqref{eq:main-6} with a non-removable singularity at the origin is radially symmetric. One may wonder whether there are other solutions and whether one can classify all the solutions. Let $t=\ln |x|$ and $\psi(t)=|x|^{\frac{n-2\sigma}{2}}u(|x|)$. Then we have
\be \label{eq:main-7}
\psi(t)=\int_{-\infty}^\infty K(t-s)\psi(s)^{\frac{n+2\sigma}{n-2\sigma}} \,\ud s,
\ee
where
\[
K(t)=\frac{1}{2^{\frac{n-2\sigma}{2}}}\int_{\mathbb{S}^{n-1}}\frac{\ud \xi}{|\cosh t-\xi_1|^{\frac{n-2\sigma}{2}}}.
\]
Non-constant periodic solutions of \eqref{eq:main-7} will be conventionally called the Fowler solutions. As mentioned earlier, Frank-K\"onig \cite{FK} classified all the singular solutions of $\Delta^2u=u^{\frac{n+4}{n-4}}$ in $\R^n\setminus\{0\}$, $n\ge 5$, corresponding to $\sigma=2$ in \eqref{eq:main-6}, where some existence results were obtained earlier by Guo-Huang-Wang-Wei \cite{GHWW}.  When $\sigma\in (0,1)$ and $n\ge 2$, the existence of Fowler solutions of $(-\Delta)^\sigma u=u^{\frac{n+2\sigma}{n-2\sigma}}$ was proved by DelaTorre-del Pino-Gonz\'alez-Wei \cite{D+}, which corresponds to $\sigma\in (0,1)$ and $n\ge 2$ in \eqref{eq:main-6}. See also Ao-DelaTorre-Gonz\'alez-Wei \cite{A+} for a gluing approach for the fractional Yamabe problem with multiple isolated singularities for $\sigma\in (0,1)$. 

Our next theorem asserts that  Fowler solutions exist for all $n\ge 1$ and $\sigma\in (0, n/2)$. 

\begin{thm} \label{thm:B} 
Suppose $n\ge 1$  and $0<\sigma<n/2$. 
There exists $T^*>0$ such that for every $T\ge T^*$, there exists a positive smooth non-constant periodic solutions of \eqref{eq:main-7} with period $T$.
\end{thm}

The proof of this existence result uses a variational method, inspired by the work DelaTorre-del Pino-Gonz\'alez-Wei \cite{D+}. One difference is that they deal directly with the fractional Laplacian operator $(-\Delta)^\sigma$ for $\sigma\in (0,1)$, while we work with its dual form, the integral equation \eqref{eq:main-7}, which allows us to unify the result for all $n\ge 1$ and $0<\sigma<n/2$. If $n\ge 2$, then the exponent $\frac{n+2\sigma}{n-2\sigma}$ in \eqref{eq:main-7} is subcritical in 1-D, and thus, we can obtain a solution by maximizing  an energy functional using compact embeddings. The $n=1$ case is a critical problem. Nevertheless, the kernel $K(t)$ has a positive mass type property near the origin, and one can still obtain a solution as in the Yamabe problem.

The classification of all solutions to \eqref{eq:main-7} remains as an open question.

Finally, we would like to extend Schoen's Harnack inequality to the integral equation setting.

\begin{thm} \label{thm:C} Suppose $n\ge 1$, $0<\sigma<n/2$, $R>0$, and $u\in C(B_{3R})$ is a positive solution of
\[
u(x)=\int_{B_{3R}}\frac{u(y)^{\frac{n+2\sigma}{n-2\sigma}}}{|x-y|^{n-2\sigma}}\,\ud y+h(x), \quad  x\in B_{3R}, \ u>0\mbox{ in }B_{3R}.
\]
 Suppose that $h\in C^1(B_{3R})$ is a positive function such that
\[
R^{\frac{n-2\sigma}{2}}\|h\|_{L^\infty(B_{5R/2})}+R\|\nabla \ln h\|_{L^\infty(B_{5R/2})} \le A
\]
for some constant $A>0$. Then
\be\label{eq:thmcharnack}
\sup_{B_R} u \cdot \inf_{B_{2R}} u\le CR^{2\sigma-n}, 
\ee
where $C>0$ depends only on $n,\sigma$ and $A$.
\end{thm}

This paper is organized as follows. In Section \ref{sec:poly}, we show that Theorem \ref{thm:higher-order} follows from Theorem \ref{thm:A}. In Section \ref{sec:ub}, we prove the upper bound and asymptotic radial symmetry in  Theorem \ref{thm:A}. In Section \ref{sec:lowerbound}, we show the lower bound in Theorem \ref{thm:A}, and complete our proof of Theorem \ref{thm:A}. In Section \ref{sec:Fowler}, we prove the existence of Fowler solutions as in Theorem \ref{thm:B}. In Section \ref{sec:harnack}, we show the Harnack inequality stated in Theorem \ref{thm:C}. In Appendix \ref{appendix:pohozaev}, we include two Pohozaev type identities that were used in proving the lower bound in Theorem \ref{thm:A}. In Appendix \ref{appendix:poisson} and Appendix \ref{appendix:technical}, we obtain some technical estimates for the Poisson extension, as well as its blow up analysis, that are used to prove the lower bound in Theorem \ref{thm:A} for the non-integer cases.

\bigskip

\noindent \textbf{Acknowledgement:} Part of this work was completed while the second named author was visiting the Department of Mathematics at the Hong Kong University of Science and Technology, to which he is grateful for providing  the very stimulating research environment and supports. Both authors would like to thank Professor YanYan Li for his interests and constant encouragement.

\section{An integral representation for the integer order cases}\label{sec:poly}

For $n\ge 3$, recall that the Green function of $-\Delta $ on the unit ball is given by
\[
G_1(x,y)=\frac{1}{(n-2)\w_{n-1}} \left(\Big|x-y\Big|^{2-n}- \Big|\frac{x}{|x|}-|x|y\Big |^{2-n}\right),
\]
where $\w_{n-1}$ is the surface area of the unit sphere in $\R^n$.  Namely, for any $u\in C^2(B_1)\cap C(\overline B_1)$,
\[
u(x)= \int_{B_1} G_1(x,y)(-\Delta) u(y) \,\ud y+ \int_{\pa B_1} H_1(x,y) u(y)\,\ud S_y,
\]
where
\[
H_1(x,y)=-\frac{\pa }{\pa \nu_y}G_1(x,y)= \frac{1-|x|^2}{\w_{n-1} |x-y|^n} \quad \mbox{for }x\in B_1, ~y\in \pa B_1.
\]
By induction, we have, for $2m<n$ and $u\in C^{2m}(B_1)\cap C^{2m-2}(\overline B_1)$,
\[
u(x)= \int_{B_1} G_m (x,y) (-\Delta )^m u(y) \,\ud y+\sum_{i=1}^{m} \int_{\pa B_1} H_i(x,y) (-\Delta)^{i-1} u(y)\,\ud S_y,
\]
where
\[
G_m(x,y) =\int_{B_1\times \cdots \times B_1} G_1(x,y_1) G_1(y_1,y_2) \dots G_1(y_{m-1}, y)\,\ud y_1\dots \ud y_{m-1},
\]
and 
\[
H_i(x,y)= \int_{B_1\times \cdots \times B_1} G_1(x,y_1) G_1(y_1,y_2) \dots G_1(y_{i-2}, y_{i-1}) H_1(y_{i-1}, y)\,\ud y_1\dots \ud y_{i-1}
\]
for $2\le i\le m$.

By direct computations, we have
\be\label{eq:Green-expand}
G_m(x,y)=c(n,m) |x-y|^{2m-n} +A_m(x,y),
\ee
$c(n,m)= \frac{\Gamma(\frac{n-2m}{2})}{2^{2m} \pi^{n/2} \Gamma(m)}$, $ A_m(\cdot,\cdot)$ is smooth in $B_1\times B_1$,  and
\[
H_i(x,y)\ge 0, \quad i=1,\dots, m.
\]

\begin{lem}\label{lem:integrable} Let $1\le m <n/2$ be an integer. Let $u\in C^{2m}(\overline B_1\setminus \{0\})$ be a positive  solution of
\be \label{eq:integer-1}
(-\Delta)^m u= u^p \quad \mbox{in }B_1\setminus \{0\},
\ee
where $p> \frac{n}{n-2m}$. Then $u(x)^p |x|^{-\al} \in L^1(B_{1})$  for every $\al<n-\frac{2mp}{p-1}$. Moreover,
\be \label{eq:dual-form}
u(x)= \int_{B_1} G_m (x,y) u(y)^p \,\ud y +\sum_{i=1}^m \int_{\pa B_1} H_i(x,y) (-\Delta)^{i-1} u(y)\,\ud S_y.
\ee
\end{lem}
\begin{proof}
The fact that $u^p\in L^1(B_1)$ was proved in Caffarelli-Gidas-Spruck  \cite{CGS}, Lin \cite{Lin} and Xu-Wei \cite{WX}. For completeness, we include a proof here, which uses some arguments in Yang \cite{YangHui}.

Let $\eta$ be a smooth function defined in $\R$ satisfying $\eta(t)= 0$ for $t\le 1$, $\eta(t)=1$ for $t\ge 2$ and $0\le \eta\le 1$ for $1\le t\le 2$. For small $\va>0$, let $\varphi_\va(x)=\eta(\va^{-1}|x|)^q$ with $q=\frac{2mp}{p-1}$. Multiplying both sides of \eqref{eq:integer-1} by $\varphi_\va(x)$  and using integration by parts, we have
\begin{align*}
\int_{B_1} u^p \varphi_\va&=   \int_{B_1} u (-\Delta)^m \varphi_\va + \int_{\pa B_1} \frac{\pa(-\Delta)^{m-1}u}{\pa\nu}\,\ud S,\\
&\le C \va^{-2m}\int_{\va\le|x|\le2\va} u \eta(\va^{-1}|x|)^{q-2m}+C\\
&\le C \va^{-2m}\int_{\va\le|x|\le2\va} u \varphi_\va^{\frac{1}{p}}+C\\
&\le C \va^{n-\frac{n}{p}-2m}\left(\int_{B_1} u^p \varphi_\va\right)^{1/p}+C\\
&\le C\left(\int_{B_1} u^p \varphi_\va\right)^{1/p}+C,
\end{align*}
from which it follows that 
\[
\int_{2\va\le|x|\le 1} u^p \le \int_{B_1} u^p \varphi_\va\le C.
\]
 By sending $\va\to 0$, we obtain 
 \[
 \int_{B_1} u^p\le C.
 \]
 
Next, we shall use some arguments in Sun-Xiong \cite{SX}. By H\"older's inequality, we have
\[
\int_{B_1} u|x|^{-\al_1}\le \Big(\int_{B_1} u^p\Big)^{1/p} \Big(\int_{B_1} |x|^{-\frac{p\al_1}{p-1}}\Big)^{(p-1)/p}<\infty
\]
if $0<\al_1<\frac{np-n}{p}$. Multiplying both sides of \eqref{eq:integer-1} by $\eta(\va^{-1}|x|) |x|^{2m-\al}$  and using integration by parts, we have
\be \label{eq:green-1}
\int_{B_1} u^p \eta(\va^{-1}|x|)  |x|^{2m-\al} =   \int_{B_1} u (-\Delta)^m \Big(\eta(\va^{-1}|x|) |x|^{2m-\al} \Big)  + \int_{\pa B_1} F(u)\,\ud S,
\ee
where $F(u)$ involves $u$ and its derivatives up to order $2m-1$. Now letting $\alpha=\alpha_1$ in \eqref{eq:green-1} and sending $\va\to 0$, we have 
\[
\int_{B_1} u^p |x|^{2m-\alpha_1}<\infty.
\]
Using H\"older's inequality again, we can show that 
\[
\int_{B_1} u|x|^{-\al_2}<\infty\quad\mbox{if}\quad \alpha_2<\frac{n(p-1)-2m}{p}+\frac{\alpha_1}{p}.
\]
Iterating this procedure, it leads to that 
\[
\int_{B_1} u|x|^{-\al_k}<\infty\quad\mbox{and}\quad\int_{B_1} u^p |x|^{2m-\alpha_k}<\infty
\]
if 
\begin{align*}
0<\al_k &<\frac{n(p-1)-2m}{p}+\frac{\alpha_{k-1}}{p}\\
&< (n(p-1)-2m) \left(\frac{1}{p}+\cdots +\frac{1}{p^{k-1}}\right)+\frac{\alpha_1}{p^{k-1}},\ k=1, 2,\cdots.
\end{align*}
We proved the first conclusion.

Let
\[
v(x)= \int_{B_1} G_m (x,y) u(y)^p \,\ud y +\sum_{i=1}^m \int_{\pa B_1} H_i(x,y) (-\Delta)^{m-i}u (y)\,\ud S_y.
\]
Let $w=u-v$. Then
\[
(-\Delta)^m w=0 \quad \mbox{in }B_1\setminus \{0\}. 
\]
Since $u^p\in L^1(B_1)$ and the Riesz potential $|x|^{2m-n}$ is weak type $\left(1,\frac{n}{n-2m}\right)$, $v\in L^{\frac{n}{n-2m}}_{weak} (B_{1})\cap L^1(B_1)$. Moreover, for every $\va>0$ we can choose $\rho>0$ such that $\int_{B_{2\rho}} u^p<\va$. Then for all sufficiently large $\lda$, we have
\[
|x\in B_{\rho}: |v(x)|>\lda| \le \left| x\in B_{\rho}: \int_{B_{2\rho}} G_m (x,y) u(y)^p \,\ud y>\lda/2 \right|\le C(n,m)\va \lda^{-\frac{n}{n-2m}}.
\]
Hence, $w\in L^{\frac{n}{n-2m}}_{weak} (B_{1})\cap L^1(B_1)$, and for every $\va>0$ there exists $\rho>0$ such that  for all sufficiently large $\lda$,
\begin{align*}
|x\in B_{\rho}: |w(x)|>\lda| &\le  |x\in B_{\rho}: |u(x)|>\lda/2| +|x\in B_{\rho}: |v(x)|>\lda/2| \\
&\le C(n,m)\va \lda^{-\frac{n}{n-2m}}.
\end{align*}
By the generalized B\^ocher's theorem for polyharmonic functions in \cite{FKM}, $(-\Delta )^m w(x)=0$ in $B_1$. Since
$w=\Delta w= \dots =\Delta^{m-1} w=0$ on $\pa B_1$, $w\equiv 0$ and thus $u=v$. This completes the proof of \eqref{eq:dual-form}.
\end{proof}

\begin{proof} [Proof of Theorem \ref{thm:higher-order} using Theorem \ref{thm:A}] 
We can suppose that $u\in C^{2m}(\overline B_1\setminus\{0\})$ and $u>0$ in $\overline B_1$, since otherwise we just consider the equation in a smaller ball.

Since $-\Delta u\ge 0$ in $B_1\setminus\{0\}$, and $u>0$ in $\overline B_1$, we know from the maximum principle that $c_1:=\inf_{B_1} u=\min_{\partial B_1}u>0$. By Lemma \ref{lem:integrable}, $u^\frac{n+2m}{n-2m}\in L^1(B_1)$. Thus, one can find $\tau<1/4$ such that
\[
\int_{B_\tau} |A_m(x,y)| u(y)^\frac{n+2m}{n-2m}\,\ud y \le \frac{c_1}{2} \quad \mbox{for }x\in B_{\tau},
\]
where $A_m(x,y)$ is as in \eqref{eq:Green-expand}. By Lemma \ref{lem:integrable}, we write
\[
u(x)= C \int_{B_\tau} \frac{u(y)^\frac{n+2m}{n-2m}}{|x-y|^{n-2m}} \,\ud y +h_1(x), 
\]
where
\begin{align*}
h_1(x)&=C\int_{B_\tau} A_m(x,y) u(y)^p\,\ud y +\int_{B_1\setminus B_\tau} G_m(x,y) u(y)^{p} \,\ud y \\
&\quad +  \sum_{i=1}^m \int_{\pa B_1} H_i(x,y) (-\Delta)^{i-1}u (y)\,\ud S_y \\
& \ge -\frac{c_1}{2} +\int_{\pa B_1} H_1(x,y) u (y)\,\ud S_y \\
&\ge -\frac{c_1}{2}+\inf_{B_1} u=\frac{c_1}{2}\quad \mbox{for }x\in B_{\tau},
\end{align*}
where we used the sign conditions \eqref{eq:h-2} in the first inequality. This is the only place that we used these sign conditions \eqref{eq:h-2}. On the other hand, $h_1$ is smooth in $B_\tau$ and satisfies $(-\Delta)^m h_1=0$ in $B_\tau$. Hence, the conclusion follows from Theorem \ref{thm:A}.
\end{proof}

Before proving Theorem \ref{thm:A} in the next sections, we would like to give two sufficient conditions for the sign assumptions \eqref{eq:h-2}.
\begin{prop}\label{prop:sufficientcondition1}
Assume as in Lemma \ref{lem:integrable}. If
\be\label{eq:blowupassumptionprop}
\lim_{\va\to 0}\int_{B_{1/2}\setminus B_\va}\frac{u(y)^{\frac{n+2m}{n-2m}}}{|y|^{n+2-2m}}\,\ud y=\infty,
\ee
then the sign conditions \eqref{eq:h-2} hold in some small punctured ball $B_\tau\setminus\{0\}$ for some $\tau>0$.
\end{prop}

\begin{proof}
By Lemma \ref{lem:integrable}, for $k=1,\cdots,m-1$, and $x\in B_1\setminus\{0\}$,
 \begin{align*}
 (-\Delta)^k u(x)&=\int_{B_1} G_{m-k} (x,y) u(y)^p \,\ud y +\sum_{i=k+1}^{m} \int_{\pa B_1} H_{i-k}(x,y) (-\Delta)^{i-1} u(y)\,\ud S_y\\
 &= C\int_{B_1} \frac{u(y)^\frac{n+2m}{n-2m}}{|x-y|^{n-2(m-k)}} \,\ud y +h_2(x), 
 \end{align*}
 where
 \[
 h_2(x)=C \int_{B_1} A_{m-k}(x,y)u(y)^\frac{n+2m}{n-2m}\,\ud y +\sum_{i=k+1}^m \int_{\pa B_1} H_{i-k}(x,y) (-\Delta)^{i-1} u(y)\,\ud S_y.
 \]
Then $h_2$ is smooth in $B_{1/2}$.

Let $r>0$ be small, which will be chosen below. Let $v=r^{\frac{n-2m}{2}}u(rx)$. Then
\begin{align*}
(-\Delta)^k v(x)&= C \int_{B_{1/r}} \frac{v(y)^\frac{n+2m}{n-2m}}{|x-y|^{n-2(m-k)}} \,\ud y +r^{\frac{n-2m}{2}+2k}h_2(rx)\\
&\ge C \int_{B_{1/r}\setminus B_1} \frac{v(y)^\frac{n+2m}{n-2m}}{|x-y|^{n-2(m-k)}} \,\ud y +r^{\frac{n-2m}{2}+2k}h_2(rx)\\
&=:\tilde h_2(x).
  \end{align*}
For every $x\in B_{1}$,
\begin{align*}
\tilde h_2(x)&\ge C \int_{B_{1/r}\setminus B_2} \frac{v(y)^\frac{n+2m}{n-2m}}{|x-y|^{n-2(m-k)}} \,\ud y +r^{\frac{n-2m}{2}+2k}h_2(rx)\\
&\ge C \int_{B_{1/r}\setminus B_2} \frac{v(y)^\frac{n+2m}{n-2m}}{|y|^{n-2(m-k)}} \,\ud y +r^{\frac{n-2m}{2}+2k}h_2(rx)\\
&=C r^{\frac{n-2m}{2}+2k}\int_{B_{1}\setminus B_{2r}} \frac{u(y)^\frac{n+2m}{n-2m}}{|y|^{n-2(m-k)}} \,\ud y +r^{\frac{n-2m}{2}+2k}h_2(rx)\\
&\ge C r^{\frac{n-2m}{2}+2k}\int_{B_{1}\setminus B_{2r}} \frac{u(y)^\frac{n+2m}{n-2m}}{|y|^{n-2(m-1)}} \,\ud y +r^{\frac{n-2m}{2}+2k}h_2(rx).
  \end{align*}
By the assumption \eqref{eq:blowupassumptionprop}, we can choose $r\in(0,1/4)$ small such that
\[
C\int_{B_{1}\setminus B_{2r}} \frac{u(y)^\frac{n+2m}{n-2m}}{|y|^{n-2(m-1)}} \,\ud y\ge \|h_2\|_{L^\infty(B_{1/4})}+1.
\]
Hence, $\tilde h_2$ is positive, and thus, $(-\Delta)^k v$ is positive in $B_1\setminus\{0\}$. So $(-\Delta)^k u$ is positive in $B_r\setminus\{0\}$.

\end{proof}

\begin{prop}\label{prop:sufficientcondition2}
Assume as in Lemma \ref{lem:integrable}. If $0$ is a non-removable singularity and 
\be\label{eq:blowupassumptionprop2}
u(x)=O(|x|^{2m-n-2})\quad\mbox{near }0,
\ee
then the sign conditions \eqref{eq:h-2} hold in some small punctured ball $B_\tau\setminus\{0\}$ for some $\tau>0$.
\end{prop}
\begin{proof}
We will show that \eqref{eq:blowupassumptionprop} holds. If not, then
\[
\int_{B_{1/2}}\frac{u(y)^{\frac{n+2m}{n-2m}}}{|y|^{n+2-2m}}\,\ud y<\infty.
\]
Then by the assumption \eqref{eq:blowupassumptionprop2}, we have
\[
\int_{B_{1/2}} u(y)^{\frac{2n}{n-2m}}\,\ud y\le C\int_{B_{1/2}}\frac{u(y)^{\frac{n+2m}{n-2m}}}{|y|^{n+2-2m}}\,\ud y<\infty.
\]
Since $u$ satisfies \eqref{eq:dual-form}, we know from the regularity result, Corollary 1.1 in \cite{Li04}, that $u$ is locally bounded in $B_1$, and thus, $0$ is removable. This is a contradiction to the assumptions.
\end{proof}

\section{The upper bound and asymptotic radial symmetry}
\label{sec:ub}

For $x\in \R^n$ and $\lda>0$, we denote
\[
\xi^{x,\lda}= x+\frac{\lda^2(\xi-x)}{|\xi-x|^2} \quad \mbox{for } \xi\neq x, \quad \om^{x,\lda}=\{\xi^{x,\lda}: \xi\in \om\},
\]
and
\[
u_{x,\lda }(\xi)= \left(\frac{\lda}{|\xi-x|}\right)^{n-2\sigma} u(\xi^{x,\lda}).
\]
Note that $(\xi^{x,\lda})^{x,\lda}=\xi$ and $(u_{x,\lda })_{x,\lda}= u$. If $x=0$, we write $u_{0,\lda }$ as $u_\lda$.

\begin{lem} \label{lem:fix_error} Suppose $f\in C^1(B_{2})$ is positive and
\[
|\nabla \ln f|\le C_0 \quad \mbox{in }B_{3/2}
\]
for some constant $C_0>0$.  Then there exists a positive constant $0<r_0<1/2$ depending only on $n,\sigma$ and $C_0$ such that for every $x\in B_1$ and $0< \lda\le r_0$ there holds
\[
f_{x,\lda}(y) \le f(y) \quad \mbox{for }|y-x|\ge \lda, ~y\in B_{3/2}.
\]

\end{lem}

\begin{proof}  For any $x\in B_1$, we have
\begin{align}
\frac{\ud }{\ud r} (r^{\frac{n-2\sigma}{2}} f(x+r\theta))&=r^{\frac{n-2\sigma}{2}-1}f (x+r\theta)\left(\frac{n-2\sigma}{2}-r \frac{\nabla f\cdot \theta}{f}\right) \nonumber \\&
\ge r^{\frac{n-2\sigma}{2}-1}f (x+r\theta)\left(\frac{n-2\sigma}{2}-C_0r\right) >0
\label{eq:start-0}
\end{align}
for all $0<r< \bar r:= \min\{\frac{1}{2}, \frac{n-2\sigma}{2C_0}\}$ and $\theta \in \mathbb{S}^{n-1} $.
For any $y\in B_{\bar r}(x)$, $0<\lda < |y-x|\le \bar r$, let $\theta=\frac{y-x}{|y-x|}$, $r_1=|y-x|$ and $r_2=\frac{\lda^2 }{|y-x|^2} r_1$. It follows from \eqref{eq:start-0} that
\[
r_2^{\frac{n-2\sigma}{2}} f(x+r_2 \theta)< r_1^{\frac{n-2\sigma}{2}} f(x+r_1 \theta).
\] Namely,
\be \label{eq:start-i0}
f_{x,\lda}(y)\le f(y), \quad 0<\lda< |y-x|\le \bar r.
\ee
Note that, for $|y-x|\ge \bar r$
\[
f_{x,\lda}(y) =\left(\frac{\lda}{|y-x|}\right)^{n-2\sigma} f(y^{x,\lda})\le \left(\frac{\lda}{\bar r}\right)^{n-2\sigma} \max_{B_{3/2}} f \le e^{\frac{3}{2}C_0} \left(\frac{\lda}{\bar r}\right)^{n-2\sigma} \inf_{B_{3/2}} f \le f(y),
\]
if we choose $\lda \le r_0$ with $e^{\frac{3}{2}C_0} (\frac{r_0}{\bar r})^{n-2\sigma}\le 1 $.

This finishes the proof.
\end{proof}

Let $u$ be a positive solution of \eqref{eq:main-3}. Replacing  $u(x)$ by $r^{\frac{n-2\sigma}{2}} u(rx)$ for $r=\frac 12$, we may consider the equation in $B_2$ for convenience, namely,   
\be \label{eq:main}
u(x)= \int_{B_2} \frac{u(y)^{\frac{n+2\sigma}{n-2\sigma}}}{|x-y|^{n-2\sigma}}\,\ud y +h(x) \quad \mbox{for } x\in B_2\setminus \{0\}
\ee with $u\in C( B_2\setminus \{0\})\cap L^\frac{n+2\sigma}{n-2\sigma}(B_2)$ and 
\be \label{eq:err-grad}
|\nabla \ln h|\le C_0 \quad \mbox{in }B_{3/2}.
\ee
If we extend $u$ to be identically $0$ outside $B_2$, then we have
\[
u(x)= \int_{\R^n} \frac{u(y)^{\frac{n+2\sigma}{n-2\sigma}}}{|x-y|^{n-2\sigma}}\,\ud y +h(y) \quad \mbox{for } x\in B_2\setminus \{0\}.
\]
Using the following two identities (see, e.g., page 162 of \cite{Li04}),
\be \label{eq:ID-1}
\left(\frac{\lda}{|\xi-x|}\right)^{n-2\sigma} \int_{|z-x|\ge \lda} \frac{u(z)^{\frac{n+2\sigma}{n-2\sigma}}}{|\xi^{x,\lda}-z|^{n-2\sigma}}\,\ud z= \int_{|z-x|\le \lda} \frac{u_{x,\lda}(z)^{\frac{n+2\sigma}{n-2\sigma}}}{|\xi-z|^{n-2\sigma}}\,\ud z
\ee
and
\be \label{eq:ID-2}
\left(\frac{\lda}{|\xi-x|}\right)^{n-2\sigma} \int_{|z-x|\le \lda} \frac{u(z)^{\frac{n+2\sigma}{n-2\sigma}}}{|\xi^{x,\lda}-z|^{n-2\sigma}}\,\ud z= \int_{|z-x|\ge \lda} \frac{u_{x,\lda}(z)^{\frac{n+2\sigma}{n-2\sigma}}}{|\xi-z|^{n-2\sigma}}\,\ud z,
\ee
one has
\[
u_{x,\lda}(\xi)=\int_{\R^n} \frac{u_{x,\lda}(z)^{\frac{n+2\sigma}{n-2\sigma}}}{|\xi-z|^{n-2\sigma}}\,\ud z +h_{x,\lda}(\xi) \quad \mbox{for } \xi\in B_2^{x,\lda}.
\]

Thus, for any $x\in B_1$ and $\lda<1$, we have,  for $\xi\in B_2 \setminus \big(\{0\} \cup  B_{\lda}(x) \big)$,
\[
u(\xi)-u_{x,\lda}(\xi)= \int_{|z-x|\ge \lda} K(x,\lda;\xi,z) (u(z)^{\frac{n+2\sigma}{n-2\sigma}}-u_{x,\lda}(z)^{\frac{n+2\sigma}{n-2\sigma}})\,\ud z+(h_{x,\lda}(\xi) -h(\xi)),
\]
where
\[
K(x,\lda;\xi,z)=\frac{1}{|\xi-z|^{n-2\sigma}}- \left(\frac{\lda}{|\xi-x|}\right)^{n-2\sigma} \frac{1}{ |\xi^{x,\lda}-z|^{n-2\sigma}}.
\]
It is elementary to check that
\[
K(x,\lda;\xi,z) >0 \quad \mbox{for all}\quad |\xi-x|>\lda>0, ~ |z-x|>\lda>0.
\]

This allows us to prove the upper bound and asymptotic radial symmetry in Theorem \ref{thm:A} using the moving sphere method introduced by Li-Zhu \cite{LZhu}. See also Li-Zhang \cite{LZhang} and Zhang \cite{Zhang} for more applications of the moving sphere method.

\begin{prop}\label{prop:upbound} Let $u\in C(B_2 \setminus \{0\})\cap L^\frac{n+2\sigma}{n-2\sigma}(B_2)$ be a positive solution of \eqref{eq:main}. Suppose $h\in C^1(B_2)$ is a positive function satisfying \eqref{eq:err-grad}.  Then
\[
\limsup_{x\to 0} |x|^{\frac{n-2\sigma}{2}} u(x) <\infty.
\]
\end{prop}

\begin{proof}
Suppose by contradiction that there exists a sequence of points $\{x_j\}_{j=1}^\infty \subset B_2$ converging to the origin such that
\[
|x_j|^{\frac{n-2\sigma}{2}}u(x_j)\to \infty\quad \mbox{as }j\to \infty.
\]
Consider
\[
v_j(x):=\left(\frac{|x_j|}{2}-|x-x_j|\right)^{\frac{n-2\sigma}{2}} u(x) \quad \mbox{for } |x-x_j| \le \frac{|x_j|}{2}.
\]
Since $u$ is positive and continuous in $\overline B_{|x_j|/2}(x_j)$, we can find a maximum point  $\bar x_j\in B_{|x_j|/2}(x_j)$ of $v_j$. That is
\[
v_j(\bar x_j)=\max_{|x-x_j|\leq \frac{|x_j|}{2}}v_j(x)>0.
\]
Let $2\mu_j:=\frac{|x_j|}{2}-|\bar x_j-x_j|>0.$
Then
\[
0<2\mu_j\leq \frac{|x_j|}{2}\quad\mbox{and}\quad \frac{|x_j|}{2}-|x-x_j|\ge\mu_j \quad \forall ~ |x-\bar x_j|\leq \mu_j.
\]
By the definition of $v_j$, we have
\be \label{eq:cl3}
(2\mu_j)^{\frac{n-2\sigma}{2}}u(\bar x_j)=v_j(\bar x_j)\ge v_j(x)\ge (\mu_j)^{\frac{n-2\sigma}{2}}u(x)\quad \forall ~ |x-\bar x_j|\leq \mu_j.
\ee
Thus, we have
\[
2^{\frac{n-2\sigma}{2}}u(\bar x_j)\ge u(x)\quad \forall ~ |x-\bar x_j|\leq \mu_j.
\]
We also have
\be\label{eq:cl4}
(2\mu_j)^{\frac{n-2\sigma}{2}}u(\bar x_j)=v_j(\bar x_j)\ge v_j(x_j)= \left(\frac{|x_j|}{2}\right)^{\frac{n-2\sigma}{2}}u(x_j)\to \infty\quad\mbox{as }i\to\infty.
\ee

Let
\[
w_j(y)=\frac{1}{u(\bar x_j)} u\left(\bar x_j +\frac{y}{u(\bar x_j)^{\frac{2}{n-2\sigma}}}\right),\quad  h_j(y)=\frac{1}{u(\bar x_j)} h\left(\bar x_j +\frac{y}{u(\bar x_j)^{\frac{2}{n-2\sigma}}}\right)  \quad \mbox{in }\om_j,
\]
where
\[
\om_j=\Big \{y\in \R^n:\bar x_j +\frac{y}{u(\bar x_j)^{\frac{2}{n-2\sigma}}} \in B_2\setminus \{0\} \Big\}.
\]
We also extend $w_j$ to be zero outside of $\Omega_j$.  Then
\begin{equation}\label{eq:wjjj}
 w_j(y) =\int_{\R^n} \frac{w_j(z)^{\frac{n+2\sigma}{n-2\sigma}}}{|y-z|^{n-2\sigma}}\,\ud z +h_j(y)  \quad \mbox{for }y\in  \om_j
\end{equation} 
and $w_j(0)=1$. Moreover, it follows from \eqref{eq:cl3} and \eqref{eq:cl4} that
\[
\|h_j\|_{C^1(\om_j)}\to 0, \quad  w_j(y)\leq 2^{\frac{n-2\sigma}{2}} \quad\mbox{in } B_{R_j},
\]
where \[R_j:=\mu_j u(\bar x_j)^{\frac{2}{n-2\sigma}}\to \infty \mbox{ as } j\to \infty.\] By the regularity results in Section 2.1 of \cite{JLX}, and the proof of Proposition 2.9 in \cite{JLX}, there exists $w>0$ such that
\[
w_j\to w \quad \mbox{in } C^{\al}_{loc}(\R^n)
\]
for some $\alpha>0$, and $w$ satisfies
\[
 w(y) =\int_{\R^n} \frac{w(z)^{\frac{n+2\sigma}{n-2\sigma}}}{|y-z|^{n-2\sigma}}\,\ud z   \quad \mbox{for }y \in \R^n.
\] 
Since $w(0)=1$, by the classification results in \cite{CLO} or \cite{Li04}, we have
\be \label{eq:cl5}
w(y)=\left(\frac{1+\mu^2|y_0|^2}{1+\mu^2|y-y_0|^2}\right)^{\frac{n-2\sigma}{2}}
\ee
for some $\mu>0$ and some $y_0\in\R^n$. 

On the other hand, we are going to show that, for every $\lda>0$
\be\label{eq:aim1}
w_{\lda}(y)\leq w(y)\quad \forall ~ |y-x|\ge\lda,
\ee
This contradicts to \eqref{eq:cl5}.

Let us arbitrarily fix $\lda_0>0$. Then for all $j$ large, we have $0<\lda_0<\frac{R_j}{10}$. Let
\[
\Sigma_j:=\Big \{y\in \R^n:\bar x_j +\frac{y}{u(\bar x_j)^{\frac{2}{n-2\sigma}}} \in B_1\setminus \{0\} \Big\} \subset \subset \om_j.
\]
We will show that for all sufficiently large $j$,
\be\label{eq:aim11}
(w_j)_{\lda_0}(y)\leq w_j(y)\quad \forall ~ |y|\ge\lda_0,\ y\in \Sigma_j.
\ee
Then \eqref{eq:aim1} follows from \eqref{eq:aim11} by sending $j\to \infty$.

By Lemma \ref{lem:fix_error}, there exist $\bar r>0$  such that, for all $0<\lda\le \bar r$ and $\bar x\in B_{\frac{1}{100}}$,
\be \label{eq:fix-er}
\left(\frac{\lda}{|y|}\right)^{n-2\sigma}h(y^{0,\lda}+\bar x) \le h(y+\bar x) \quad \mbox{for }|y|\ge \lda, ~y\in B_{149/100}.
\ee
Let $j$ large such that  $\lda_0 u(\bar x_j)^{-\frac{1}{n-2\sigma}} <\bar r$. It follows that, for every $0<\lda\le \lda_0$
\be \label{eq:moving-f}
(h_j)_{\lda}(y) \le h_j(y) \quad \mbox{for }y\in \Sigma_j \backslash B_{\lda}.
\ee

\emph{Claim 1.} There exists a positive real number $\lda_1$ independent of (large) $j$ such that for any $0<\lda<\lda_1$, we have
\[
(w_j)_{\lda }(y)\leq w_j(y) \quad \text{in } \Sigma_j \backslash B_{\lda}.
\]

Since $w_j\to w$ locally uniformly and $w$ is given in \eqref{eq:cl5}, we know that $w_j\ge c_0>0$ on $B_1$ for all $j$ sufficiently large. On the other hand, from the equation \eqref{eq:wjjj} and the regularity results in Section 2.1 in \cite{JLX}, we know that $|\nabla w_j|\le C_0<\infty$ on $B_1$ for all $j$ sufficiently large. By the proof of \eqref{eq:start-i0}, there exists a $r_{0}>0$ independent of $j$ (large) such that for all $0<\lda\le r_{0} $
\be \label{eq:start-i}
(w_j)_{\lda}(y)<w_j(y), \quad 0<\lda< |y|\le r_{0}.
\ee
Again, since $w_j\ge c_0>0$ on $B_1$ for all $j$ sufficiently large, we have
\[
w_j(x) \ge c_0^{\frac{n+2\sigma}{n-2\sigma}}\int_{B_1} |x-y|^{2\sigma-n}\,\ud y\ge \frac{1}{C} (1+|x|)^{2\sigma-n} \quad \mbox{in }\om_j
\]
for some constant $C>0$. 
It follows that we can find a sufficiently small $0<\lda_1\le r_{0}$ such that for all $0<\lda<\lda_1$
\[
(w_j)_{\lda}(y) \le \left(\frac{\lda_1}{|y|}\right)^{n-2\sigma} \max_{B_{r_0}} w_j \le C \left(\frac{\lda_1}{|y|}\right)^{n-2\sigma}  \le w_j(y) \quad \mbox{for all }|y|\ge r_0,\ y\in\om_j.
\]
Together with \eqref{eq:start-i}, we proved the Claim 1.

We define
\[
\bar \lda:=\sup \{0<\mu\le \lda_0\ |\ (w_j)_{\lda}(y)\leq w_{j}(y),\ \forall~|y-x_0|\geq \lda, ~y\in \Sigma_j,~\forall~ 0<\lda <\mu\},
\]
where $\lda_0$ is fixed at the beginning. By Claim 1, $\bar\lda$ is well defined.

\medskip

\emph{Claim 2}: $\bar\lda=\lda_0$ for all sufficiently large $j$.

\medskip

By \eqref{eq:ID-1} and \eqref{eq:ID-2}, we have, for any $\bar \lda\le \lda \le \bar \lda+\frac12$,  $y\in \Sigma_j$ with $|y|> \lda$,
\be\label{eq:movingsphere}
\begin{split}
&w_j(y)-(w_j)_{ \lda}(y)\\&= \int_{B_{ \lda}^c} K(0, \lda; y,z) \Big(w_j(z)^{\frac{n+2\sigma}{n-2\sigma}}-(w_j)_{ \lda}(z)^{\frac{n+2\sigma}{n-2\sigma}} \Big)\,\ud z+h_j(y)-(h_j)_{ \lda}(y)\\&
\ge \int_{\Sigma_j \setminus B_{ \lda }} K(0, \lda; y,z) \Big(w_j(z)^{\frac{n+2\sigma}{n-2\sigma}}-(w_j)_{ \lda}(z)^{\frac{n+2\sigma}{n-2\sigma}} \Big)\,\ud z +J( \lda, w_j,y),
\end{split}
\ee
where we have used \eqref{eq:moving-f} and
\begin{align}
J(\lda, w_j,y)&= \int_{\R^n \setminus \Sigma_j } K(0, \lda; y,z) \Big(w_j(z)^{\frac{n+2\sigma}{n-2\sigma}}-(w_j)_{ \lda}(z)^{\frac{n+2\sigma}{n-2\sigma}} \Big)\,\ud z  \nonumber\\
& =\int_{\om_j \setminus \Sigma_j } K(0, \lda; y,z) \Big(w_j(z)^{\frac{n+2\sigma}{n-2\sigma}}-(w_j)_{ \lda}(z)^{\frac{n+2\sigma}{n-2\sigma}} \Big)\,\ud z \nonumber\\
&\quad - \int_{\om_j^c} K(0, \lda; y,z) (w_j)_{ \lda}(z)^{\frac{n+2\sigma}{n-2\sigma}}  \,\ud z.\label{eq:definitionofJ}
\end{align}
For $z\in \R^n\setminus \Sigma_j$ and $\bar \lda \le \lda \le \bar \lda+1$, we have $|z| \ge \frac 12 u(\bar x_j)^{\frac{2}{n-2\sigma}}$ and thus
\[
(w_j)_{\lda}(z) \le \left(\frac{\lda }{|z|}\right)^{n-2\sigma} \max_{B_{\bar \lda+1}} w_j \le C u(\bar x_j)^{-2}.
\]
Since $u\ge c_0>0$ in $B_2\setminus B_{1/2}$, by the definition of $w_j$, we have
\be\label{eq:wj-l-1}
w_j(y)\ge \frac{c_0}{u(\bar x_j)} \quad \mbox{in }\om_j\setminus \Sigma_j.
\ee
It follows that for large $j$,
\[
 w_j(z)^{\frac{n+2\sigma}{n-2\sigma}}-(w_j)_{ \lda}(z)^{\frac{n+2\sigma}{n-2\sigma}} \ge \frac{1}{2}  w_j(z)^{\frac{n+2\sigma}{n-2\sigma}}\quad\mbox{in }\Omega_j\setminus\Sigma_j.
\]
Then, we claim that
\begin{align}
J(\lda, w_j,y) &\ge \frac{1}{2} \left(\frac{c_0 }{u(\bar x_j)}\right)^{\frac{n+2\sigma}{n-2\sigma}} \int_{\om_j \setminus \Sigma_j } K(0, \lda; y,z) \,\ud z- C \int_{\om_j^c} K(0, \lda; y,z) \left(\frac{\lda }{|z|}\right)^{n+2\sigma}   \,\ud z  \nonumber  \\&
\ge \begin{cases}
\frac{1}{C} (|y|-\lda) u(\bar x_j)^{-1},& \quad \mbox{if }\lda\le |y|\le \bar \lda +1,\\
\frac{1}{C}u(\bar x_j)^{-1}, & \quad \mbox{if } |y|>\bar \lda +1, y\in\Sigma_j.
\end{cases}
\label{eq:lower-J}
\end{align}
Indeed, since $ K(0, \lda; y,z)=0$  for $|y|=\lda$ and
\[
y\nabla_y \cdot K(0, \lda; y,z)\Big|_{|y|=\lda}= (n-2\sigma)|y-z|^{2\sigma-n-2}(|z|^2-|y|^2)>0
\]
for $|z|\ge  \bar \lda+2$, and using the positivity and smoothness of $K$, we have
\be \label{eq:lower-K}
\frac{\delta_1}{|y-z|^{n-2\sigma}}(|y|-\lda) \le K(0, \lda; y,z) \le \frac{\delta_2}{|y-z|^{n-2\sigma}}(|y|-\lda)
\ee
for $\bar \lda \le \lda \le |y|\le \bar \lda+1,~\bar\lda+2\le |z|\le \Lda<\infty,$ where $0<\delta_1\le \delta_2<\infty$. Also, if $\Lda$ is large enough, then
\[
0<c_0\le y\cdot \nabla_y (|y-z|^{n-2\sigma}K(0, \lda; y,z))\le C_0<\infty \quad\mbox{for all }|z|\ge\Lda, \ \bar \lda \le \lda \le |y|\le \bar \lda+1.
\]
Hence, \eqref{eq:lower-K} holds for $\bar \lda \le \lda \le |y|\le \bar \lda+1,~|z|\ge \Lda$ as well.

On the other hand, by the definition of $K(0, \lda; y,z)$, one can verify that for $|y|\ge \bar \lda+1$ and $|z|\ge\bar\lda+2$, 
\[
\frac{\delta_3}{|y-z|^{n-2\sigma}}\le K(0, \lda; y,z) \le \frac{1}{|y-z|^{n-2\sigma}}
\]
for some $0<\delta_3<1$. 

Therefore, for large $j$, $\lda \le |y|\le \bar \lda+1$ (recall that $\lda\le \bar \lda+\frac12$), we have
\begin{align*}
J(\lda, w_j,y)  &\ge \frac{1}{2} \left(\frac{c_0}{u(\bar x_j)}\right)^{\frac{n+2\sigma}{n-2\sigma}}  \int_{\om_j\setminus \Sigma_j}\frac{\delta_1}{|y-z|^{n-2\sigma}}(|y|-\lda) \,\ud z\\&
\quad - C \int_{\om_j^c} \frac{\delta_2}{|y-z|^{n-2\sigma}}(|y|-\lda)  \left(\frac{\lda }{|z|}\right)^{n+2\sigma}   \,\ud z\\&
\ge \frac{1}{C_1} (|y|-\lda) u(\bar x_j)^{-1}- \frac{1}{C_2} (|y|-\lda) u(\bar x_j)^{-\frac{2n}{n-2\sigma}} \\
&\ge \frac{1}{2 C_1} (|y|-\lda) u(\bar x_j)^{-1}
\end{align*}
and, for $|y|\ge\bar \lda+1$, $y\in \Sigma_j$, similarly, we have
\[
J(\lda, w_j,y)\ge \frac{1}{C_3} u(\bar x_j)^{-1}- \frac{1}{C_4} u(\bar x_j)^{-\frac{2n}{n-2\sigma}}\ge \frac{1}{2 C_3} u(\bar x_j)^{-1},
\]
where $C_1,C_2,C_3,C_4$ are positive constants and we have used $u(\bar x_j)\to \infty$. Hence, \eqref{eq:lower-J} is verified.

By \eqref{eq:movingsphere} and \eqref{eq:lower-J}, there exists $\va_1\in (0,1/2)$ (which depends on $j$) such that
\[
w_j(y)-(w_j)_{\bar \lda}(y) \ge \frac{\va_1}{|y|^{n-2\sigma}} \quad \forall~|y|\ge \bar \lda+1,\ y\in\Sigma_j.
\]
By the above inequality  and the explicit formula for $(w_{j})_{\lda}$, there exists $0<\va_2<\va_1$ such that
\begin{align}
w_j(y)-(w_j)_{\lda}(y)& \ge  \frac{\va_1}{|y|^{n-2\sigma}} +((w_j)_{\bar \lda}(y)-(w_j)_{\lda}(y)) \nonumber \\&
\ge  \frac{\va_1}{2|y|^{n-2\sigma}} \quad \forall~|y|\ge \bar \lda+1, ~\bar \lda\le \lda\le \bar\lda+\va_2.
\label{eq:movingsphere-2}
\end{align}
For $\va\in (0,\va_3]$ which we choose below, by \eqref{eq:movingsphere} and \eqref{eq:lower-J} we have, for $\bar \lda\le \lda\le \bar\lda+\va$ and for $\lda\le |y|\le \bar \lda+1$,
\begin{align*}
w_j(y)-(w_j)_{\lda}(y)&
\ge \int_{\lda\le|z|\le \bar \lda+1} K(0, \lda; y,z) \Big(w_j(z)^{\frac{n+2\sigma}{n-2\sigma}}-(w_j)_{ \lda}(z)^{\frac{n+2\sigma}{n-2\sigma}} \Big)\,\ud z\\&
\quad +\int_{ \bar\lda+2\le|z|\le \bar \lda+3} K(0, \lda; y,z) \Big(w_j(z)^{\frac{n+2\sigma}{n-2\sigma}}-(w_j)_{ \lda}(z)^{\frac{n+2\sigma}{n-2\sigma}} \Big)\,\ud z \\& \ge 
- C\int_{\lda \le |z|\le \lda+\va } K(0, \lda; y,z) (|z|-\lda)\,\ud z\\& \quad  +
\int_{\lda+\va \le|z|\le \bar \lda+1} K(0, \lda; y,z) \Big((w_j)_{ \bar \lda}(z)^{\frac{n+2\sigma}{n-2\sigma}}-(w_j)_{ \lda}(z)^{\frac{n+2\sigma}{n-2\sigma}} \Big)\,\ud z\\&
\quad +\int_{\bar\lda+2\le|z|\le \bar \lda+3} K(0, \lda; y,z) \Big(w_j(z)^{\frac{n+2\sigma}{n-2\sigma}}-(w_j)_{\lda}(z)^{\frac{n+2\sigma}{n-2\sigma}} \Big)\,\ud z, 
\end{align*}
where we have used 
 \[
|(w_j)(z)^{\frac{n+2\sigma}{n-2\sigma}}-(w_j)_{\lda}(z)^{\frac{n+2\sigma}{n-2\sigma}} | \le C(|z|-\lda)
\]
 in the second inequality. 
Because of \eqref{eq:movingsphere-2}, there exists $\delta_5>0$ such that
\[
w_j(z)^{\frac{n+2\sigma}{n-2\sigma}}-(w_j)_{ \lda}(z)^{\frac{n+2\sigma}{n-2\sigma}}  \ge \delta_5 \quad \mbox{for }
\bar \lda+2\le |z|\le \bar \lda+3.
\]
Since $\|w_j\|_{C^1(B_2)}\le C$ (independent of $j$), it follows that there exists some constant $C>0$ independent of $\va$ such that for $\bar \lda \le \lda \le \bar \lda+\va$,
\[
|(w_j)_{\bar \lda}(z)^{\frac{n+2\sigma}{n-2\sigma}}-(w_j)_{\lda}(z)^{\frac{n+2\sigma}{n-2\sigma}} | \le C(\lda-\bar \lda)\le C\va \quad \forall~  \lda\le |z|\le\bar\lda+1
\]
and for, $\lda\le |y|\le \bar \lda+1$, 
\begin{align*}
\int_{\lda+\va \le|z|\le \bar \lda+1} K(0, \lda; y,z)\,\ud z &\le \left|\int_{\lda+\va\le|z|\le \bar \lda+1}\left(\frac{1}{|y-z|^{n-2\sigma}}-\frac{1}{|y^{0,\lda}-z|^{n-2\sigma}}\right)\,\ud z \right| \\
&\quad+\int_{\lda+\va\le|z|\le \bar \lda+1} \left|\left(\frac{\lda}{|y|}\right)^{n-2\sigma}-1\right|\frac{1}{|y^{0,\lda}-z|^{n-2\sigma}}\,\ud z\\
&\le 
\begin{cases}
C(\va^{2\sigma-1}+1)|y^{0,\lda}-y|+C(|y|-\lda)\mbox{ if }\sigma\neq 1/2\\
C(|\ln\va|+1)|y^{0,\lda}-y|+C(|y|-\lda)\mbox{ if }\sigma= 1/2
\end{cases}\\
&\le C(\va^{2\sigma-1}+|\ln\va|+1)(|y|-\lda), 
\end{align*}
and  
\begin{align*}
\int_{\lda\le|z|\le  \lda+\va } K(0, \lda; y,z) (|z|-\lda)\,\ud z  
&\le \left|\int_{\lda\le|z|\le \lda+\va}\left(\frac{|z|-\lda}{|y-z|^{n-2\sigma}}-\frac{|z|-\lda}{|y^{0,\lda}-z|^{n-2\sigma}}\right)\,\ud z \right| \\
&\quad+\va\int_{\lda\le|z|\le \lda+\va} \left|\left(\frac{\lda}{|y|}\right)^{n-2\sigma}-1\right|\frac{1}{|y^{0,\lda}-z|^{n-2\sigma}}\,\ud z\\
&\le I +C\va(|y|-\lda),
\end{align*}
where
\begin{align*}
I&=\left|\int_{\lda\le|z|\le \lda+\va}\left(\frac{|z|-\lda}{|y-z|^{n-2\sigma}}-\frac{|z|-\lda}{|y^{0,\lda}-z|^{n-2\sigma}}\right)\,\ud z \right|\\
&\le C\va(\va^{2\sigma-1}+|\ln\va|+1)(|y|-\lda) \quad\mbox{if}\quad |y|\ge \lda+10\va.\end{align*}
When $\lda<|y|\le \lda+10\va$, then
\begin{align*}
I&\le \left|\int_{\lda\le|z|\le \lda+10(|y|-\lda)}\left(\frac{|z|-\lda}{|y-z|^{n-2\sigma}}-\frac{|z|-\lda}{|y^{0,\lda}-z|^{n-2\sigma}}\right)\,\ud z \right|\\
&\quad+\left|\int_{\lda+10(|y|-\lda)\le|z|\le \lda+\va}\left(\frac{|z|-\lda}{|y-z|^{n-2\sigma}}-\frac{|z|-\lda}{|y^{0,\lda}-z|^{n-2\sigma}}\right)\,\ud z \right|\\
&\le C (|y|-\lda)\int_{\lda\le|z|\le \lda+10(|y|-\lda)}\left(\frac{1}{|y-z|^{n-2\sigma}}+\frac{1}{|y^{0,\lda}-z|^{n-2\sigma}}\right)\,\ud z\\
&\quad+ C|y-y^{0,\lda}| \int_{\lda+10(|y|-\lda)\le|z|\le\lda+\va }\frac{|z|-\lda}{|y-z|^{n-2\sigma+1}}\,\ud z\\
&\le C (|y|-\lda) \sup_{w\in\R^n }\int_{\lda\le|z|\le \lda+100\va}\frac{1}{|w-z|^{n-2\sigma}}\,\ud z\\
&\le C (|y|-\lda) \va^{\frac{2\sigma}{n}}.
\end{align*}

It follows that for $\lda<|y|\le\bar\lda+1$, using \eqref{eq:lower-K},
\begin{align*}
&w_j(y)-(w_j)_{\lda}(y)\\&\ge - C\va^{\frac{2\sigma}{n}} (|y|-\lda)+\delta_1\delta_5 (|y|-\lda) \int_{ \lda+2\le|z|\le \bar \lda+3}\frac{1}{|y-z|^{n-2\sigma}} \,\ud z \\&
\ge \Big(\delta_1\delta_5 c-C\va^{\frac{2\sigma}{n}}\Big)(|y|-\lda) \ge 0
\end{align*}
if $\va$ is sufficiently small. This and \eqref{eq:movingsphere-2} contradict to the definition of $\bar \lda$ if $\bar \lda<\lda_0$. Therefore, Claim 2 follows.

Therefore, \eqref{eq:aim1} is proved and the proof of Proposition \ref{prop:upbound} is completed.
\end{proof}

We remark that the above arguments also apply to subcritical cases.

\begin{prop}\label{prop:symmetry} 
Assume as in Proposition \ref{prop:upbound}. Then
\[
u(x)=\bar u(|x|)(1+O(|x|)) \quad \mbox{as }x\to 0.
\]

\end{prop}

\begin{proof} We are going to show that there exists $0<\va <\min\{1/10,\bar r\}$ such that
\be \label{eq:aim2}
u_{x,\lda}(y) \le u(y) \quad \mbox{for }y\in  B_1\setminus B_{\lda}(x),~ 0<\lda<|x|<\va,
\ee
where $\bar r$ is the one such that  \eqref{eq:fix-er} hold for all $0<\lda\le \bar r$.

First of all, by the proof of \eqref{eq:start-i0}, for every $x\in B_{1/10}\setminus\{0\}$ there exists a $0<r_{x}<|x|$ such that for all $0<\lda\le r_{x} $
\[
u_{x,\lda}(y)<u(y), \quad 0<\lda< |y-x|\le r_{x}.
\]
By the equation \eqref{eq:main}, we have
\begin{equation}\label{eq:pointwiselowerbound}
u(x)\ge 4^{2\sigma-n}\int_{B_2} u^{\frac{n+2\sigma}{n-2\sigma}}(y)\,\ud y=:c_0>0,
\end{equation}
and thus, we can find $0<\lda_1\ll r_x$ such that, for every $0<\lda\le \lda_1$,
\[
u_{x,\lda}(y)<u(y), \quad  y\in B_2\setminus (B_{r_x}(x)\cup \{0\}).
\]
Combining the above two inequalities together, we have, for every $0<\lda\le \lda_1$
\[
u_{x,\lda}(y)<u(y), \quad  y\in B_2\setminus (B_{r_x}(x)\cup \{0\}).
\]
This ensures that
\[
\bar\lda(x):= \sup \{0<\mu\le |x| \ |\  u_{x,\lda}(y)\leq u(y),\quad \forall~y\in B_2\setminus (B_{\lda}(x)\cup \{0\}),~\forall~ 0<\lda <\mu\}
\]
is well defined and is positive. 

We are going to show that there exists $\va>0$ such that  $\bar \lda=|x|$ for all $|x|\le \va$. For brevity, we will denote $\bar \lda=\bar\lda(x)$ in the below.

For every $\bar \lda \le  \lda <|x|\le \bar r$, we have, for $y\in B_1$,
\begin{align*}
u(y)-u_{x,\lda}(y)\ge \int_{B_1\setminus B_{\lda}(x)}K(x,\lda; y,z) \Big(u(z)^{\frac{n+2\sigma}{n-2\sigma}}-u_{x,\lda}(z)^{\frac{n+2\sigma}{n-2\sigma}} \Big)\,\ud z+J(\lda, u, y),
\end{align*}
where
\begin{align*}
J(\lda,u,y)= &\int_{B_2\setminus B_{1}}K(x,\lda; y,z) \Big(u(z)^{\frac{n+2\sigma}{n-2\sigma}}-u_{x,\lda}(z)^{\frac{n+2\sigma}{n-2\sigma}} \Big)\,\ud z \\&
-\int_{B_2^c}  K(x,\lda; y,z) u_{x,\lda}(z)^{\frac{n+2\sigma}{n-2\sigma}} \,\ud z.
\end{align*}
For $y\in  B_1^c$ and $\lda<|x|<\va<1/10$, we have
\[
\left|x+\frac{\lda^2 (y-x)}{|y-x|^2}\right| \ge |x|-\frac{10}{9} \lda^2 \ge  |x|-\frac{10}{9} |x|^2 \ge \frac{8}{9}|x|.
\]
It follows from Proposition \ref{prop:upbound} that
\[
u\left(x+\frac{\lda^2 (y-x)}{|y-x|^2}\right) \le C |x|^{-\frac{n-2\sigma}{2}}.
\]
Thus for all $y\in B_1^c$,
\be \label{eq:asy-1}
\begin{split}
u_{x,\lda}(y) &= \left(\frac{\lda}{|y-x|}\right)^{n-2\sigma} u\left(x+\frac{\lda^2 (y-x)}{|y-x|^2}\right)\\
&\le C\lda^{n-2\sigma} |x|^{-\frac{n-2\sigma}{2}} \le C |x|^{\frac{n-2\sigma}{2}} \le C \va^{\frac{n-2\sigma}{2}}.
\end{split}
\ee
By \eqref{eq:pointwiselowerbound}, we have
\[
u_{x,\lda}(y) <u(y) \quad \mbox{for }y\in B_2\setminus B_1.
\]

For $y\in B_1\setminus (B_\lda(x)\cup\{0\})$, by \eqref{eq:pointwiselowerbound} and \eqref{eq:asy-1}, using the proof of \eqref{eq:lower-J}, we have
\begin{align}
J(\lda,u,y)&\ge \int_{B_2\setminus B_{1}}K(x,\lda; y,z) \Big(c_0^{\frac{n+2\sigma}{n-2\sigma}}- C^{\frac{n+2\sigma}{n-2\sigma}} \va^{n+2\sigma}\Big)\,\ud z\nonumber \\&
\quad  -C\int_{B_2^c}  K(x,\lda; y,z) \left(\Big(\frac{|x|}{|z-x|}\Big)^{n-2\sigma} |x|^{-\frac{n-2\sigma}{2}}\right)^{\frac{n+2\sigma}{n-2\sigma}}\,\ud z \nonumber \\&
 \ge \frac{1}{2} c_0^{\frac{n+2\sigma}{n-2\sigma}}  \int_{B_2\setminus B_{1}}K(x,\lda; y,z)  \,\ud z- C \va^{\frac{n+2\sigma}{2}} \int_{B_2^c}  K(x,\lda; y,z) \frac{1}{|z-x|^{n+2\sigma}}\,\ud z \nonumber\\&
  \ge \frac{1}{2} c_0^{\frac{n+2\sigma}{n-2\sigma}}  \int_{B_{19/10}\setminus B_{11/10}}K(0,\lda; y-x,z)  \,\ud z \nonumber \\
  &\quad-C \va^{\frac{n+2\sigma}{2}} \int_{B_{19/10}^c}  K(0,\lda; y-x,z) \frac{1}{|z|^{n+2\sigma}}\,\ud z \nonumber\\&
 \ge \frac{1}{C_1} (|y-x|-\lda),
 \nonumber
\end{align}
if we let $\va$ be sufficiently small, where $C_1>0$ is a constant independent of $x$. If $\bar \lda<|x|$, using the proof of Proposition \ref{prop:upbound}, we will arrive at a contradiction to the definition of $\bar \lda$. The proof is identical and we omit  it here. Therefore, \eqref{eq:aim2} is proved.

Let $r$ be small (less than $\va^2$), $x_1,x_2\in \pa B_r$ be such that
\[
u(x_1)=\max_{\pa B_r} u, \quad u(x_2)=\min_{\pa B_r} u.
\]
Let $x_3=x_1+\frac{\va(x_1-x_2)}{4|x_1-x_2|}$ and $\lda=\sqrt{\frac{\va}{4}(|x_1-x_2|+\frac{\va}{4})}$.  It follows from \eqref{eq:aim2} that
\[
u_{x_3,\lda }(x_2)\le u(x_2).
\]
Notice that
\begin{align*}
u_{x_3,\lda }(x_2)&= \left(\frac{\lda}{|x_1-x_2|+\frac{\va}{4}}\right)^{n-2\sigma}u(x_1)\\&
=\left(\frac{1}{4|x_1-x_2|/\va+1}\right)^{\frac{n-2\sigma}{2}}u(x_1)\\
&\ge (\frac{1}{8r/\va+1})^{\frac{n-2\sigma}{2}}u(x_1).
\end{align*}
Thus
\[
\max_{\pa B_r} u \le (8r/\va+1)^{\frac{n-2\sigma}{2}} \min_{\pa B_r} u.
\]
The proposition follows immediately.
\end{proof}

\section{The lower bound and removability of the singularity}
\label{sec:lowerbound}

One consequence of the upper bound in Proposition \ref{prop:upbound} is the following Harnack inequality and derivative estimates, which will be used very frequently in the rest of the paper.

\begin{lem}\label{lem:harnack}
Assume as in Proposition \ref{prop:upbound}. Then for all $0<r<1/4$, we have
\begin{equation}\label{eq:harnack}
\sup_{B_{3r/2}\setminus B_{r/2}} u\le C \inf_{B_{3r/2}\setminus B_{r/2}} u, 
\end{equation}
where $C$ is a positive constant independent of $r$.
\end{lem}
\begin{proof}
Let $v(y)= r^{\frac{n-2\sigma}{2}} u(ry)$. Then
\begin{equation}\label{eq:rescaledequation}
v(y)= \int_{B_{2/r}} \frac{v(z)^{\frac{n+2\sigma}{n-2\sigma}}}{|y-z|^{n-2\sigma }}\,\ud z+h_r(y), 
\end{equation}
where $h_r(y)= r^{\frac{n-2\sigma}{2}} h(ry)$. By Proposition \ref{prop:upbound}, we have $v\le C$ in $B_{2}\setminus B_{1/10}$. For $|x|=1$, let 
\[
g_x(y)=\int_{B_{2/r}\setminus B_{9/10}(x)} \frac{v(z)^{\frac{n+2\sigma}{n-2\sigma}}}{|y-z|^{n-2\sigma }}\,\ud z.
\]
Then for any $y_1,y_2\in B_{1/2}(x)$, we have
\begin{align*}
g_x(y_1)&=\int_{B_{2/r}\setminus B_{9/10}(x)} \frac{|y_2-z|^{n-2\sigma }}{|y_1-z|^{n-2\sigma }}\frac{v(z)^{\frac{n+2\sigma}{n-2\sigma}}}{|y_2-z|^{n-2\sigma }}\,\ud z\\
&\le C \int_{B_{2/r}\setminus B_{9/10}(x)} \frac{v(z)^{\frac{n+2\sigma}{n-2\sigma}}}{|y_2-z|^{n-2\sigma }}\,\ud z\le Cg_x(y_2).
\end{align*}
Hence, $g$ satisfies the Harnack inequality in $B_{1/2}(x)$. Since $h$ also satisfies the Harnack inequality in $B_{1/2}(x)$ and
\[
v(y)=\int_{B_{9/10}(x)} \frac{v(z)^{\frac{n+2\sigma}{n-2\sigma}}}{|y-z|^{n-2\sigma }}\,\ud z+g_x(y)+ h_r(y)\quad\mbox{in } B_{1/2}(x), 
\]
it follows from Theorem 2.3 in \cite{JLX} that
\[
\sup_{B_{1/2}(x)} v \le C \inf_{B_{1/2}(x)} v.
\]
A covering argument leads to 
\[
\sup_{1/2\le |y|\le 3/2} v \le C \inf_{1/2\le |y|\le 3/2 } v.
\]
We complete the proof by rescaling back to $u$.
\end{proof}

\begin{lem} \label{lem:estimates} Assume as in Proposition \ref{prop:upbound}. Suppose also that $h$ is smooth in $B_2$. Then we have  
\[
|\nabla ^k u(x)| \le C(k) |x|^{-\frac{n-2\sigma}{2}-k}, \quad \forall~ x \in B_1\setminus \{0\},
\]
for all $k=0,1,\dots$, where $C(k) >0$ is independent of $x$. 
\end{lem} 
\begin{proof} For any fixed $x$, let $r=|x|$ and $v(y)= r^{\frac{n-2\sigma}{2}} u(ry)$. Then we have $v$ satisfies \eqref{eq:rescaledequation}. By Proposition \ref{prop:upbound}, we have $v\le C$ in $B_{3/2}\setminus B_{1/2}$. By the local estimates (in Section 2.1 of  \cite{JLX}), and making use of the smoothness of $h$, we have $|\nabla^k v(x)| \le C$ for $|x|=1$. Scaling back to $u$, the proposition is proved. 
\end{proof}

The Harnack inequality will lead to that $\liminf_{x\to 0}u(x)=\infty$ if $0$ is not a removable singularity.
\begin{lem}\label{prop:infinity}
Assume as in Proposition \ref{prop:upbound}, and suppose $h$ is smooth in $B_1$. If 
\[
\limsup_{x\to 0}u(x)=\infty,
\]
then
\[
\liminf_{x\to 0}u(x)=\infty.
\]
\end{lem}
\begin{proof}
Let $\{x_i\}$ be such that $r_i=|x_i|\to 0$ and $u(x_i)\to\infty$. By Lemma \ref{lem:harnack}, $$\inf_{\partial B_{r_i}}u\to\infty.$$

If $\sigma\ge 1$, then we have $-\Delta (u-h)\ge 0$ in $B_2\setminus\{0\}$. Hence, $$\min_{B_{r_i}\setminus B_{r_{i+1}}}(u-h)=\min_{\partial B_{r_i}\cup\partial B_{r_{i+1}}} (u-h).$$ Since $h$ is smooth in $B_1$, we have $\min_{B_{r_i}\setminus B_{r_{i+1}}}u\to\infty$.

If $\sigma< 1$, we let
\[
v(x)=\int_{B_2} \frac{u(y)^{\frac{n+2\sigma}{n-2\sigma}}}{|x-y|^{n-2\sigma}}\,\ud y \quad\mbox{for all }x\in\R^n\setminus\{0\}.
\]
Then $v>0$ in $\R^n$. By Proposition \ref{prop:upbound}, we have $v(x)=u(x)-h(x)\le C|x|^{-\frac{n-2\sigma}{2}}$ when $|x|$ is small and $u^{\frac{n+2\sigma}{n-2\sigma}}\in L^p(B_1)$ for all $1\le p<\frac{2n}{n+2\sigma}$. Consequently, $v(x)\le C|x|^{2\sigma-n}$ when $|x|$ is large. Therefore, $(-\Delta)^\sigma v$ is well-defined in $B_2\setminus\{0\}$, and 
$$(-\Delta)^\sigma v(y)=c_0u(y)^{\frac{n+2\sigma}{n-2\sigma}}=c_0 u^{\frac{4\sigma}{n-2\sigma}}v +  c_0 u^{\frac{4\sigma}{n-2\sigma}} h>0 \quad\mbox{in}\quad B_2\setminus\{0\},$$
where $c_0$ is a positive constant. Let $V$ be the extension of $v$ in $\R^{n+1}_+=\{(x,t): x\in\R^n, t>0\}$ defined in \eqref{eq:poissonextension}. Then
\[
\begin{cases}
\mathrm{div} (t^{1-2\sigma}\nabla_{x,t} V)=0\quad\mbox{in }\R^{n+1}_+\\
-\lim_{t\to 0}t^{1-2\sigma}\partial_tV(x,t)=c_1 u^{\frac{4\sigma}{n-2\sigma}}v +  c_1 u^{\frac{4\sigma}{n-2\sigma}} h\quad\mbox{on }B_2\setminus\{0\}
\end{cases}
\]
for some positive constant $c_1$. Let $V_j=V(r_jx,r_jt)$, $u_j(x)=u(r_jx), h_j(x)=h(r_jx), v_j(x)=v(r_jx)$. Then
\[
\begin{cases}
\mathrm{div} (t^{1-2\sigma}\nabla_{x,t} V_j)=0\quad\mbox{in }\R^{n+1}_+\\
-\lim_{t\to 0}t^{1-2\sigma}\partial_tV_j(x,t)=c_0 r_j^{2\sigma} u_j^{\frac{4\sigma}{n-2\sigma}}v_j +  c_0 r_j^{2\sigma}u_j^{\frac{4\sigma}{n-2\sigma}} h_j\quad\mbox{on }B_{2/r_j}\setminus\{0\}.
\end{cases}
\]
By Proposition \ref{prop:upbound}, we know $r_j^{2\sigma}u_j^{\frac{4\sigma}{n-2\sigma}} $ is locally bounded. By the Harnack inequality (Proposition 2.6 in \cite{JLX0}), we have
\[
\sup_{\partial''\mathcal{B}^+} V_j\le C(\inf_{\partial''\mathcal{B}^+} V_j+1),
\]
where $\partial''\mathcal{B}^+=\{(x,t): t>0, |x|^2+t^2=1\}$. Therefore, by the maximum principle,
\[
\inf_{r_{j+1}\le |x|\le r_j} v(x)\ge \inf_{r_{j+1}\le \sqrt{|x|^2+t^2}\le r_j} V(x,t)= \inf_{\sqrt{|x|^2+t^2}= r_j, r_{j+1}} V(x,t)\ge c\min_{\partial B_{r_i}\cup\partial B_{r_{i+1}}} v-C.
\]
Since $v=u-h$ on $B_1\setminus\{0\}$ and $h$ is bounded on $B_1$, we have $\min_{B_{r_i}\setminus B_{r_{i+1}}}u\to\infty$.
\end{proof}

We can improve Lemma \ref{prop:infinity} to the following.

\begin{lem}\label{lem:remove}
Assume as in Proposition \ref{prop:upbound}. If 
\[
\lim_{x\to 0}|x|^{\frac{n-2\sigma}{2}} u(x)=0,
\]
then $u$ can be extended as a continuous function near the origin $0$.
\end{lem}
\begin{proof}
Here, we will use some barrier functions that were used in \cite{JLX}. 

For every $\delta>0$, choose $\tau>0$ small such that
\[
u(x)\le \delta |x|^{-\frac{n-2\sigma}{2}}\quad\mbox{in }B_\tau\setminus\{0\}.
\]
Arbitrarily fix $\va>0$ and $\mu \in (0,\frac{n-2\sigma}{2})$. Let $M$ be a large positive constant to be chosen in the below. Define
\[
f(x)=
\begin{cases}
M|x|^{-\mu}+\va |x|^{2\sigma-n-\mu},\quad 0<|x|<\tau\\
u(x),\quad \tau<|x|<2.
\end{cases}
\]
Note that for every $0<\mu<n-2\sigma$ and $0<|x|<2$, we have
\begin{align*}
\int_{\R^n} \frac{1}{|x-y|^{n-2\sigma}|y|^{\mu+2\sigma}}\,\ud y&= |x|^{2\sigma-n} \int_{\R^n}
\frac{1}{||x|^{-1}x-|x|^{-1}y|^{n-2\sigma}|y|^{\mu+2\sigma}}\,\ud y \\&
= |x|^{-\mu+2\sigma} \int_{\R^n} \frac{1}{||x|^{-1}x-z|^{n-2\sigma}|z|^{\mu+2\sigma}}\,\ud z \\&
 \le C\Big( \frac{1}{n-2\sigma-\mu}+\frac{1}{\mu} +1\Big)|x|^{-\mu},
\end{align*}
where we did the change of variables $y=|x|z$. Hence, for $0<|x|<2$,
\begin{align*}
\int_{B_\tau}\frac{u^{\frac{4\sigma}{n-2\sigma}}(y)f(y)}{|x-y|^{n-2\sigma}}\,\ud y&\le \delta^{\frac{4\sigma}{n-2\sigma}}\int_{\R^n}\frac{f(y)}{|x-y|^{n-2\sigma}|y|^{2\sigma}}\,\ud y\\
&\le C\delta^{\frac{4\sigma}{n-2\sigma}} f(x)\\
&<\frac{1}{2}f(x)\quad\mbox{for sufficiently small } \delta.
\end{align*}
For $0<|x|<\tau$, $\bar x=\tau x/|x|$, we have
\begin{align*}
\int_{B_2\setminus B_\tau}\frac{u^{\frac{4\sigma}{n-2\sigma}}(y)f(y)}{|x-y|^{n-2\sigma}}\,\ud y&= \int_{B_2\setminus B_\tau}\frac{|\bar x-y|^{n-2\sigma}}{|x-y|^{n-2\sigma}}\frac{u^{\frac{n+2\sigma}{n-2\sigma}}(y)}{|\bar x-y|^{n-2\sigma}}\,\ud y\\
&\le 2^{n-2\sigma}\int_{B_2\setminus B_\tau}\frac{u^{\frac{n+2\sigma}{n-2\sigma}}(y)}{|\bar x-y|^{n-2\sigma}}\,\ud y\\
&\le 2^{n-2\sigma} u(\bar x)\\
&\le 2^{n-2\sigma}\max_{\partial B_\tau}u.
\end{align*}
Then for $0<|x|<\tau$, 
\[
h(x)+ \int_{B_2}\frac{u^{\frac{4\sigma}{n-2\sigma}}(y)f(y)}{|x-y|^{n-2\sigma}}\,\ud y\le h(x)+2^{n-2\sigma}\max_{\partial B_\tau}u+\frac{1}{2}f(x)<f(x)
\]
if $M\ge \max_{\partial B_\tau}u$ is sufficiently large.

Since $u(x)\le \delta |x|^{-\frac{n-2\sigma}{2}}$ in $B_\tau\setminus\{0\}$, by the definition of $f$, there exists $\rho\in (0,\tau)$ (depending on $\va$) such that $f\ge u$ in $B_\rho\setminus\{0\}$, and also $f>u$ near $\partial B_\tau$. We claim that $u(x)\le f(x)$ in $B_\tau\setminus\{0\}$. If not, let
\[
\bar t:=\inf\{t>1: tf_i>u\quad\mbox{in }B_\tau\setminus\{0\}\}.
\]
The we have $1<\bar t<\infty$, and there exist $\bar x\in B_\tau\setminus\overline B_\rho$ such that $\bar t f(\bar x)=u(\bar x)$. Since, for $0<|x|<\tau$,
\[
\bar tf(x)\ge \int_{B_2}\frac{u^{\frac{4\sigma}{n-2\sigma}}(y)\bar tf(y)}{|x-y|^{n-2\sigma}}\,\ud y+\bar th(x)\ge \int_{B_2}\frac{u^{\frac{4\sigma}{n-2\sigma}}(y)\bar tf(y)}{|x-y|^{n-2\sigma}}\,\ud y+h(x),
\]
we have
\[
\bar tf(x)-u(x)\ge  \int_{B_2}\frac{u^{\frac{4\sigma}{n-2\sigma}}(y)(\bar tf(y)-u(y))}{|x-y|^{n-2\sigma}}\,\ud y.
\]
This leads to a contradiction by evaluating at $\bar x$. Hence 
\[
u(x)\le f(x)\le M|x|^{-\mu}+\va |x|^{2\sigma-n-\mu} \quad\mbox{in }B_\tau\setminus\{0\}.
\]
Notice that $\tau$ does not depend on $\va$. Then by sending $\va\to 0$, we prove
\[
u(x)\le M|x|^{-\mu}\quad\mbox{in }B_\tau\setminus\{0\}.
\]
Hence $u^{\frac{4\sigma}{n-2\sigma}}\in L^p$ for some $p>\frac{n}{2\sigma}$. By the regularity result, Corollary 2.2 in \cite{JLX}, we have that $u$ is continuous in $B_\tau$.
\end{proof}

Now we are ready to prove the lower bound in Theorem \ref{thm:A}.

\begin{prop}\label{prop:zero}
 Assume as in Proposition \ref{prop:upbound}. Suppose in addition that $h$ is smooth in $B_2$,  $h$ is a positive function in $\R^n$ satisfying \eqref{eq:assumptionsh} if $\sigma$ is not an integer, and $h$ satisfies  $(-\Delta)^\sigma h=0$ in $B_2$.   If 
 \[
\liminf_{x\to 0}|x|^{\frac{n-2\sigma}{2}} u(x)=0,
\]
then 
\[
\lim_{x\to 0}|x|^{\frac{n-2\sigma}{2}} u(x)=0,
\]
and consequently, $u$ can be extended as a continuous function near the origin $0$.

\end{prop}

\begin{proof} We suppose by contradiction that $\limsup_{x\to 0} |x|^{\frac{n-2\sigma}{2}} u(x)=C>0 $. It follows from Lemma \ref{prop:infinity} that
\[
\liminf_{|x|\to 0} u(x)=\infty.
\]
 By the assumption and the Harnack inequality in Lemma \ref{lem:harnack},  we can find $r_i\to 0$ such that
$r_i^{\frac{n-2\sigma}{2}}\bar u(r_i) \to 0$ and $r_i$ is a local minimum point of $r^{\frac{n-2\sigma}{2}}\bar u(r)$, where $\bar u(r)$ is the spherical average of $u$ over $\pa B_r$. Let
\[
\varphi_i(y)=\frac{u(r_i y)}{u(r_i e_n)}.
\]
By \eqref{eq:main},
\[
\varphi_i(y)=\int_{B_{2/r_i} } \frac{(r_i^{\frac{n-2\sigma}{2}} u(r_ie_n) )^{\frac{4\sigma}{n-2\sigma}} \varphi_i(z)^{\frac{n+2\sigma}{n-2\sigma}}}{|y-z|^{n-2\sigma}}\,\ud z +h_i(y), \quad y\in B_{2/r_i}\setminus \{0\},
\]
where $h_i(y)=u(r_i e_n)^{-1} h(r_i y)$.

We claim that
\be \label{eq:converg2}
 \varphi_i (y) \to \frac{1}{2|y|^{n-2\sigma}} + \frac12 \quad \mbox{in }C^2_{loc}(\R^n \setminus \{0\}).
\ee

First, since $u(r_i e_n)\to\infty$, we have that $h_i (y) \to 0$ in $C^n_{loc}(\R^n)$.

Second, using the Harnack inequality in Lemma \ref{lem:harnack}, we know  that $r_i^{\frac{n-2\sigma}{2}} u(r_ie_n)\to 0$ and $ \varphi_i$ is locally uniformly bounded in $B_{2/r_i}\setminus\{0\}$. Hence, $(r_i^{\frac{n-2\sigma}{2}} u(r_ie_n) )^{\frac{4\sigma}{n-2\sigma}} \varphi_i(y)^{\frac{n+2\sigma}{n-2\sigma}} \to 0$ in $C_{loc}^n(\R^n\setminus \{0\})$. Thus, we have, for any $t>1$,  $0<|y|<t$, $0<\va<|y|/100$, after passing to subsequences, 
\begin{align*}
&\lim_{i\to\infty}\int_{B_t}\frac{(r_i^{\frac{n-2\sigma}{2}} u(r_ie_n) )^{\frac{4\sigma}{n-2\sigma}} \varphi_i(z)^{\frac{n+2\sigma}{n-2\sigma}}}{|y-z|^{n-2\sigma}}\,\ud z\\
&\quad=  
\lim_{i\to\infty}\int_{B_\va}\frac{(r_i^{\frac{n-2\sigma}{2}} u(r_ie_n) )^{\frac{4\sigma}{n-2\sigma}} \varphi_i(z)^{\frac{n+2\sigma}{n-2\sigma}}}{|y-z|^{n-2\sigma}}\,\ud z\\
&\quad=|y|^{2\sigma-n}(1+O(\va))\lim_{i\to\infty}\int_{B_\va}(r_i^{\frac{n-2\sigma}{2}} u(r_ie_n) )^{\frac{4\sigma}{n-2\sigma}} \varphi_i(z)^{\frac{n+2\sigma}{n-2\sigma}}.
\end{align*}
By sending $\va\to 0$, we have
\[
\lim_{i\to\infty}\int_{B_t}\frac{(r_i^{\frac{n-2\sigma}{2}} u(r_ie_n) )^{\frac{4\sigma}{n-2\sigma}} \varphi_i(z)^{\frac{n+2\sigma}{n-2\sigma}}}{|y-z|^{n-2\sigma}}\,\ud z= \frac{a}{|y|^{n-2\sigma}},
\]
for some constant $a\ge 0$. 

Also,
\[
\int_{B_{2/r_i} \setminus B_t}
\frac{(r_i^{\frac{n-2\sigma}{2}} u(r_ie_n) )^{\frac{4\sigma}{n-2\sigma}} \varphi_i(z)^{\frac{n+2\sigma}{n-2\sigma}}}{|y-z|^{n-2\sigma}}\,\ud z  \to f(y)\ge 0\quad \mbox{in }C_{loc}^n(B_t)
\]
for some function $f\in C^2(B_t)$, since the left-hand side is locally uniformly bounded in $C_{loc}^{n+1}(B_t)$. Moreover, for any fixed large $R>0$ and $y\in B_t$, we have
\[
\int_{t \le |y|\le R}
\frac{(r_i^{\frac{n-2\sigma}{2}} u(r_ie_n) )^{\frac{4\sigma}{n-2\sigma}} \varphi_i(z)^{\frac{n+2\sigma}{n-2\sigma}}}{|y-z|^{n-2\sigma}}\,\ud z  \to 0\quad\mbox{as } i\to \infty,
\]
and for any $y', y''\in B_t$, we have
\[
\begin{split}
\int_{B_{2/r_i}  \setminus B_R}
&\frac{(r_i^{\frac{n-2\sigma}{2}} u(r_ie_n) )^{\frac{4\sigma}{n-2\sigma}} \varphi_i(z)^{\frac{n+2\sigma}{n-2\sigma}}}{|y'-z|^{n-2\sigma}}\,\ud z \\
&\le  \left( \frac{R+t}{R-t} \right)^{n-2\sigma}\int_{B_{2/r_i}  \setminus B_R}
\frac{(r_i^{\frac{n-2\sigma}{2}} u(r_ie_n) )^{\frac{4\sigma}{n-2\sigma}} \varphi_i(z)^{\frac{n+2\sigma}{n-2\sigma}}}{|y''-z|^{n-2\sigma}}\,\ud z  .
\end{split}
\]
Therefore, it follows that
\[
f(y') \le \left( \frac{R+t}{R-t} \right)^{n-2\sigma} f(y'').
\]
By sending $R\to \infty$ and exchanging the roles of $y'$ and $y''$, we have $f(y'') = f(y')$. Thus,
\[
f(y)\equiv f(0)\quad \mbox{for all } y\in B_t.
\]
Since $ \varphi_i$ is locally uniformly bounded in $B_{2/r_i}\setminus\{0\}$, we know that it is locally uniformly bounded in $C^{n+1}(B_{2/r_i}\setminus\{0\})$. Hence, subject to a subsequence, $\varphi_i\to\varphi$ in $C^n_{loc}(\R^n\setminus\{0\})$ for some $\varphi$ as $i\to\infty$. From the above arguments, we conclude that
\[
\varphi_i(y) \to \frac{a}{|y|^{n-2\sigma}}+f(0) \quad \mbox{in }C^n_{loc}(\R^n\setminus \{0\}). 
\]
 Since $\varphi_i(e_n)=1$ and
\[
\frac{\mathrm{d}}{\mathrm{d}r}\left\{ r^{\frac{n-2\sigma}{2}}\bar \varphi_i(r)\right\}\Big|_{r=1}=
r_i^{-\frac{n-2\sigma}{2}+1}u_i(r_i e_n)^{-1}\frac{\mathrm{d}}{\mathrm{d}r}\left\{ r^{\frac{n-2\sigma}{2}}\bar u(r)\right\} \Big|_{r=r_i}=0,
\]
we have, by  sending $i$ to $\infty$, that
\[a=f(0)= \frac{1}{2} .\]
Therefore, the claim of the limit \eqref{eq:converg2} is the proved.

Let
\begin{equation}\label{eq:definitionofv}
v(x)= \int_{B_2} \frac{u(y)^{\frac{n+2\sigma}{n-2\sigma}}}{|x-y|^{n-2\sigma}}\,\ud y +h(x)\quad\mbox{in } \R^n\setminus \{0\}.
\end{equation}
Then $v=u$ in $B_2\setminus\{0\}$,  and
\[
v(x)= \int_{B_2} \frac{v(y)^{\frac{n+2\sigma}{n-2\sigma}}}{|x-y|^{n-2\sigma}}\,\ud y +h(x)\quad\mbox{in } \R^n\setminus\{0\}.
\]
Consequently,
\begin{equation}\label{eq:algebraequation}
(-\Delta)^\sigma v(y)=c(n,\sigma) v(y)^{\frac{n+2\sigma}{n-2\sigma}}\quad\mbox{in }B_2\setminus\{0\}.
\end{equation}
for some constant $c(n,\sigma)>0$, where we used the assumption that $(-\Delta)^\sigma h=0$ in $B_2$.

 We are going derived a contradiction by a Pohozaev identity of \eqref{eq:algebraequation}, in which we need to use the algebraic structure of the differential equation \eqref{eq:algebraequation}. The needed Pohozaev identities are provided in the Appendix \ref{appendix:pohozaev}. 

We first consider the case that $\sigma$ is an integer. Let $P(u,r)$ be as in \eqref{eq:pohozaevinteger}. From Proposition \ref{prop:pohozaev}, we know that $P(u,r)$ is a constant independent of $r$. Since $|D^k \varphi_i|\le C$ near $\partial B_1$ and $r_i^{\frac{n-2\sigma}{2}}u(r_ie_n)=o(1)$, we have
\[
|D^k u(x)|\le C r_i^{-k} u(r_ie_n)=o(1)r_i^{-\frac{n-2\sigma}{2}-k}\quad\mbox{for all }|x|=r_i, k=1,2,\cdots,n.
\]
It follows that
\[
\lim_{i\to\infty}P(u,r_i)=0.
\]
Hence,
\[
P(u,r_i)=0\quad\mbox{for all } i.
\]
This implies that
\[
\int_{\partial B_1}g_\sigma(\varphi_i,\varphi_i)+ \frac{n-2\sigma}{n} c(n,\sigma)(r_i^{\frac{n-2\sigma}{2}} u(r_ie_n) )^{\frac{4\sigma}{n-2\sigma}}\int_{\partial B_1} \varphi_i^{\frac{2n}{n-2\sigma}}=0,
\]
where $g_\sigma$ is given in \eqref{eq:gpohozaev}. Sending $i\to\infty$ and multiplying $4$ on both sides, we obtain
\[
\int_{\partial B_1}g_\sigma\left(\frac{1}{|x|^{n-2\sigma}}+1,\frac{1}{|x|^{n-2\sigma}}+1\right)=0.
\]
Since 
\begin{equation}\label{eq:bubblelambda}
\bar u_\lda(x)=\left(\frac{\lda}{|x|^2+\lda^2}\right)^{\frac{n-2\sigma}{2}}
\end{equation}
satisfies
\[
(-\Delta)^\sigma \bar u_\lda=\tilde c(n,\sigma) \bar u_\lda^{\frac{n+2\sigma}{n-2\sigma}}\quad\mbox{in }\R^n
\]
for any $\lda>0$, we have
\[
P(\bar u_\lda,1)=\lim_{r\to\infty} P(\bar u_\lda,r)=0\quad\forall\ \lda>0.
\]
It follows that
\[
\int_{\partial B_1}g_\sigma(\lda^{-\frac{n-2\sigma}{2}}\bar u_\lda,\lda^{-\frac{n-2\sigma}{2}}\bar u_\lda)+ \frac{n-2\sigma}{n} \tilde c(n,\sigma)\lda^{2\sigma-n}\int_{\partial B_1} \bar u_\lda^{\frac{2n}{n-2\sigma}}=0.
\]
Sending $\lda\to 0$, we obtain
\[
\int_{\partial B_1}g_\sigma\left(\frac{1}{|x|^{n-2\sigma}},\frac{1}{|x|^{n-2\sigma}}\right)=0.
\]
Then
\begin{align*}
0&=\int_{\partial B_1}g_\sigma\left(\frac{1}{|x|^{n-2\sigma}}+1,\frac{1}{|x|^{n-2\sigma}}+1\right)-\int_{\partial B_1}g_\sigma\left(\frac{1}{|x|^{n-2\sigma}},\frac{1}{|x|^{n-2\sigma}}\right)\\
&=2(-1)^\sigma\frac{n-2\sigma}{2} \int_{\pa B_1} \frac{\pa \Delta^{\sigma-1} (|x|^{2\sigma-n})}{\pa\nu}\\
&\neq 0,
\end{align*}
where we used \eqref{eq:pohozaevleft} in the second equality. 

This reaches a contradiction. 

If $\sigma$ is not an integer, the idea is the same and it only involves technical issues from the localization formula of the fractional Laplacian $(-\Delta)^\sigma$.  The technical calculation details are given in Appendix \ref{appendix:poisson} and Appendix \ref{appendix:technical}. 

We start from Proposition \ref{prop:pohozaev22} and Lemma \ref{lem:estimatepohozaev}. It follows from Proposition \ref{prop:pohozaev22} and Lemma \ref{lem:estimatepohozaev} that
\[
P(v,r_i)=0\quad\mbox{for all } i,
\]
where $P$ is the one defined in \eqref{eq:pohozaevnoninteger}. Rewrite the above equation for $\varphi_i$ and send $i\to\infty$, we will have
\[
0= \int_{\pa'' \B_1^+} \widetilde Q_m\left(\frac{1}{2|X|^{n-2\sigma}}+\frac 12\right),
\]
where $\int_{\pa'' \B_1^+} \widetilde Q_m(U)$ is given in \eqref{eq:tildeQ}. Let $\bar u_\lda$ be as in \eqref{eq:bubblelambda}, and $\overline U_\lda=\mathcal P_\sigma*\bar u_\lda$ be as defined in \eqref{eq:poissonextension}. Then we have $P(\bar u_\lda,1)=0$ $\forall~\lda>0$, and thus,
\[
\int_{\pa'' \B_1^+} \widetilde Q_m\left(\frac{1}{2|X|^{n-2\sigma}}\right)=0.
\]
But 
\begin{align*}
0&=\int_{\pa'' \B_1^+} \widetilde Q_m\left(\frac{1}{2|X|^{n-2\sigma}}+\frac 12\right)-\int_{\pa'' \B_1^+} \widetilde Q_m\left(\frac{1}{2|X|^{n-2\sigma}}\right)\\
&=\frac{n-2\sigma}{8} \int_{\pa'' \B_r^+} t^b\pa_\nu \Delta_b^m (|X|^{2\sigma-n})\neq 0,
\end{align*}
which is a contradiction.

Hence, $\limsup_{x\to 0} |x|^{\frac{n-2\sigma}{2}} u(x)=0$. Consequently, by Lemma \ref{lem:remove}, $u$ can be extended as a continuous function near the origin $0$.
\end{proof}


\begin{proof}[Proof of Theorem \ref{thm:A}]
It follows from Propositions \ref{prop:upbound}, \ref{prop:symmetry} and \ref{prop:zero}.
\end{proof}

\section{Existence of Fowler solutions}\label{sec:Fowler}

Suppose that $u(|x|)$ is a positive solution of
\[
u(x)= \int_{\R^n} \frac{u(y)^{\frac{n+2\sigma}{n-2\sigma }}}{|x-y|^{n-2\sigma}}\,\ud y \quad \mbox{in }\R^n\setminus \{0\}.
\]
Let $U(t)=r^{\frac{n-2\sigma}{2}} u(r)$, where $r=|x|$, $\theta =\frac{x}{|x|}$ and $t=\ln r$. By changing variables $s=\ln\rho$, we have
\begin{align*}
U(t)&=e^{\frac{n-2\sigma}{2}t} \int_{0}^\infty \int_{\mathbb{S}^{n-1}}\frac{u(\rho)^{\frac{n+2\sigma}{n-2\sigma}}}{|e^{t} e_1 -\rho \zeta|^{n-2\sigma}} \rho^{n-1}\,\ud \zeta\,\ud \rho\\&
=\int_{-\infty}^\infty  \int_{\mathbb{S}^{n-1}}\frac{U(s)^{\frac{n+2\sigma}{n-2\sigma}}}{|e^{\frac{t-s}{2}} e_1 -e^{\frac{s-t}{2}} \zeta|^{n-2\sigma}} \,\ud \zeta\,\ud s\\
&=\int_{-\infty}^{\infty}K(t-s) U(s)^{\frac{n+2\sigma}{n-2\sigma}}\,\ud s,
\end{align*}
where $e_1=(1,0,\cdots,0)$ and
\begin{align*}
K(t)&= \int_{\mathbb{S}^{n-1}}\frac{1}{|e^{\frac{t}{2}} e_n -e^{\frac{-t}{2}} \zeta|^{n-2\sigma}} \,\ud \zeta\\&
=\int_{\mathbb{S}^{n-1}}\frac{1}{|e^{t}+e^{-t} -2 \zeta_1|^\frac{n-2\sigma}{2}} \,\ud \zeta.
\end{align*}
Hence,
\begin{equation}\label{eq:KKK}
K(t)=
\begin{cases}
\displaystyle \frac{1}{2^{\frac{1-2\sigma}{2}}}\left(\frac{1}{|\cosh (t)-1|^{\frac{1-2\sigma}{2}}}+\frac{1}{|\cosh (t)+1|^{\frac{1-2\sigma}{2}}}\right)\quad\mbox{if }n=1,\\
\displaystyle\frac{1}{2^{\frac{n-2\sigma}{2}}}  \w_{n-2} \int_{-1}^1 \frac{(1-\zeta_1^2)^{\frac{n-3}{2}}}{|\cosh (t)-\zeta_1|^{\frac{n-2\sigma}{2}}}\,\ud \zeta_1 \quad\mbox{if }n\ge 2,
\end{cases}
\end{equation}
where $\omega_0=2$. It is clear that $K$ is positive, even and $K(t)=O(e^{-\frac{n-2\sigma}{2}t})$ as $|t|\to \infty$. It is also elementary to check that (or see Lemma 2.6 in \cite{D+} with the $\gamma$ there replaced by $-\sigma$ and the proof is identical) 
\[
K(t)\sim |t|^{2\sigma-1}\ \mbox{ as } \ |t|\to 0, \mbox{ if } \sigma\in (0,1/2).
\]
If $\sigma=\frac 12$, then
\[
K(t)\sim |\ln|t||\ \mbox{ as } \ |t|\to 0.
\]
If $\sigma>\frac 12$, then $K$ is bounded and H\"older continuous on $\R$. 

Let $w(t)=U(t)^{\frac{n+2\sigma}{n-2\sigma}}$. We have
\begin{equation}\label{eq:periodicsolution}
w(t)^{\frac{n-2\sigma}{n+2\sigma}}= \int_{-\infty}^{\infty}K(t-s) w(s)\,\ud s.
\end{equation}
We will show that there exists a family of non-constant periodic solutions to \eqref{eq:periodicsolution}.

For any $T>0$, we consider the variational problem
\begin{equation}\label{eq:JT}
J_T= \sup_{f\in \mathcal{A}_T\setminus \{0\}} J_T[f],
\end{equation}
where
\be\label{eq:JTf}
J_T[f]= \frac{\int_0^T \int_0^T K_T(t-s) f(t)f(s)\,\ud t\ud s}{(\int_0^T |f(t)|^{\frac{2n}{n+2\sigma}}\ud t)^{\frac{n+2\sigma}{n}}},
\ee
\be\label{eq:KT}
K_T(t)= \sum_{j=-\infty}^\infty K(t+jT),
\ee
and
\[
\mathcal{A}_T= \{f: f\in L^{\frac{2n}{n+2\sigma}}_{loc}(\R), ~ f(t)=f(t+T)\quad  \mbox{for }t\in \R \}.
\]


If $\sigma>1/2$, then $K$ is bounded, and thus, by H\"older inequality,
\[
J_T[f]\le C(T) \frac{(\int_0^T |f(t)|\,\ud t)^2}{(\int_0^T |f(t)|^{\frac{2n}{n+2\sigma}}\ud t)^{\frac{n+2\sigma}{n}}}\le C(T)T^{\frac{n-2\sigma}{n}}.
\]
Therefore, $J_T<\infty$ for every $T>0$.

If $\sigma=\frac 12$, then $n\ge 2$, and $K(t)\le C(n) |t|^{-\frac 14}$. By the Hardy-Littlewood-Sobolev inequality \cite{Lieb83} and H\"older inequality that 
\[
J_T[f]\le C(T) \frac{(\int_0^T |f(t)|\,\ud t)^2+ (\int_0^T |f(t)|^{8/7}\,\ud t)^{7/4}}{(\int_0^T |f(t)|^{\frac{2n}{n+1}}\ud t)^{\frac{n+1}{n}}}\le C(T)(T^{\frac{n-1}{n}}+T^{\frac{3n-4}{4}}).
\]
Therefore, $J_T<\infty$ for every $T>0$.

If $\sigma<1/2$, then it follows from the Hardy-Littlewood-Sobolev inequality and H\"older inequality that $J_T<\infty$ for every $T>0$.

We will show that $J_T$ is always achieved by a smooth positive function. Conceptually, this is because if $n\ge 2$ then the exponent $\frac{n+2\sigma}{n-2\sigma}$ is subcritical in 1-D. When $n=1$, this is a critical problem, and we will not have the compact embedding. Nevertheless, the kernel $K$ in \eqref{eq:KKK} has a positive mass type property: $K(t)\ge |t|^{2\sigma-1}+A_0$ for some $A_0>0$ near the origin, and therefore, we still can use a variational approach to show the existence of maximizers as in the Yamabe problem.

\begin{prop} \label{prop:fowler} 
Suppose $n\ge 2$. For every $T>0$, $J_T$ is achieved by a smooth positive function. Moreover, there exists $T^*>0$ such that if  $T>T^*$, then the maximizers of $J_T$ are not constants. Consequently, there exists a family of non-constant periodic solutions to \eqref{eq:periodicsolution}.
\end{prop}
\begin{proof} 
Let $\{f_i\}$ be a maximizing sequence such that 
\[
\|f_i\|_{L^{\frac{2n}{n+2\sigma}}([0,T])}=1\quad\mbox{and}\quad\int_0^T \int_0^T K_T(t-s) f_i(t)f_i(s)\,\ud t\ud s\to J_T\quad\mbox{as }i\to\infty.
\] 
Since $J_T[|f_i|]\ge J_T[f_i]$, we can assume that all $f_i$ are nonnegative. Subject to a subsequence, we assume that $f_i\rightharpoonup f$ weakly in $L^{\frac{2n}{n+2\sigma}}([0,T])$. 

We will show that $F_i(t):= \int_0^T K_T(t-s)f_i(s)\,\ud s$ converges strongly in $L^{\frac{2n}{n-2\sigma}}([0,T])$. 

First, since $n\ge 2$, it follows from the Hardy-Littlewood-Sobolev inequality and H\"older inequality that $\|F_i\|_{L^{p}([0,T])}\le C$ for some $p>\frac{2n}{n-2\sigma}$, where $C>0$ is independent of $i$. Consequently, by H\"older inequality, for $\va>0$, there exists $\delta>0$ such that 
\[
\|F_i\|_{L^{\frac{2n}{n-2\sigma}}([0,\delta]\cup[T-\delta,T])}<\va.
\]
Secondly, we extend $F_i$ to be zero outside of $(0,T)$ and let $\tau_h F_i(x)=F_i(x+h)$. Since $s,t\in [0,T]$, when defining $F_i$, one can also consider that $f_i=0$ in $\R\setminus [0,T]$ and $K_T(t)=0$ in $\R\setminus [-T,T]$. Suppose $h>0$ small. Given $\va>0$, by Young's inequality,
\begin{align*}
\|\tau_h F_i-F_i\|_{L^{\frac{2n}{n-2\sigma}}([0,T-h])}&= \|(\tau_h K_T-K_T)*f_i\|_{L^{\frac{2n}{n-2\sigma}}([0,T-h])}\\
&\le \|(\tau_h K_T-K_T)\|_{L^{\frac{n}{n-2\sigma}}([-T,T])}\|f_i\|_{L^{\frac{2n}{n+2\sigma}}([0,T])}.
\end{align*}
Since $K_T\in L^{\frac{n}{n-2\sigma}}([-T,T])$  ($n\ge 2$), there exists $\tilde\delta>0$ such that if $h\le\tilde\delta$ then $\|(\tau_h K_T-K_T)\|_{L^{\frac{n}{n-2\sigma}}([-T,T])}\le \va$. This implies that for $\va>0$, there exists $\tilde\delta>0$ small such that if $0\le h<\tilde\delta$, then 
\begin{align*}
\|\tau_h F_i-F_i\|_{L^{\frac{2n}{n-2\sigma}}([0,T])}=\|\tau_h F_i-F_i\|_{L^{\frac{2n}{n-2\sigma}}([0,T-h])}+\|F_i\|_{L^{\frac{2n}{n-2\sigma}}([T-h,T])}<2\va.
\end{align*}
The same would hold for $-\delta< h<0$ by similar arguments.

By the Kolmogorov-M. Riesz-Fr\'echet theorem (see Theorem 4.26 in the book \cite{Brezis} of Brezis), the above implies that there is a subsequence of $\{F_i\}$, which is still denoted as $\{F_i\}$, such that $\{F_i\}$ converges strongly in $L^{\frac{2n}{n-2\sigma}}([0,T])$ to some $F$.
Hence,
\begin{align*}
\int_0^T \int_0^T K_T(t-s) f_i(t)f_i(s)\,\ud t\ud s&= \int_0^T \left(\int_0^T K_T(t-s) f_i(s)\,\ud s-F(t)\right)(f_i(t)-f(t))\,\ud t\\
&\quad + \int_0^T F(t)(f_i(t)-f(t))\,\ud t\\
&\quad+ \int_0^T \int_0^T K_T(t-s) f_i(s)f(t)\,\ud t\ud s\\
&\to \int_0^T \int_0^T K_T(t-s) f(t)f(s)\,\ud t\ud s\ \mbox{strongly as }i\to\infty.
\end{align*}
Then $f\not\equiv 0$. Since $\|f\|_{L^{\frac{2n}{n+2\sigma}}([0,T])}\le\liminf_{i\to\infty}\|f_i\|_{L^{\frac{2n}{n+2\sigma}}([0,T])}=1$, we have
\[
J_T[f]=J_T.
\]
Also, $f$ satisfies the equation
\[
c_0f(t)^{\frac{n-2\sigma}{n+2\sigma}}= \int_{0}^{T}K_T(t-s) f(s)\,\ud s
\]
for some positive constant $c_0$.  Since the kernel $K_T$ is positive, it is clear from the above equation that $f>0$ on $[0,T]$.  

For convenience, we substitute $f(t)^{\frac{n-2\sigma}{n+2\sigma}}$ by $g(t)$, and obtain
\[
c_0g(t)= \int_{0}^{T}K_T(t-s) g(s)^{\frac{n+2\sigma}{n-2\sigma}}\,\ud s.
\]

If $\sigma>\frac 12$, then $K_T$ is H\"older continuous on $[0,T]$. Since $g$ and $K_T$ are periodic, using bootstrap arguments it follows that $g$ is smooth.

Let $\eta$ be an even smooth cut-off function so that $\eta=1$ in $[-2,2]$ and is zero in $\R\setminus[-3,3]$.

 If $\sigma<\frac 12$, then
\begin{align*}
c_0g(t)&= \int_{-\infty}^{+\infty}K(t-s) g(s)^{\frac{n+2\sigma}{n-2\sigma}}\,\ud s\\
&= \int_{-3}^{3}K(t-s) \eta(s)g(s)^{\frac{n+2\sigma}{n-2\sigma}}\,\ud s+\int_{|s|\ge 2}K(t-s) (1-\eta(s))g(s)^{\frac{n+2\sigma}{n-2\sigma}}\,\ud s\\
&=:h_1(t)+h_2(t).
\end{align*}
Since $g\in L^{\frac{2n}{n-2\sigma}}(0,1)$, we may suppose that $g(t_0)<\infty$ for some $t_0\in (0,1)$. Since $K(t-s)\le C K(t_0-s)$ for all $t\in (-1,1)$ and all $|s|\ge 3$, where $C$ is independent of $t$ and $s$, we know that $h_2\in L^\infty(-1,1)$.
Since $g^{\frac{4\sigma}{n-2\sigma}}\in L^{\frac{1}{2\sigma}}_{loc}(\R)$ and $K(t)\sim |t|^{2\sigma-1}$ as $|t|\to 0$, it follows from the regularity result Theorem 1.3 in \cite{Li04} that $g\in L^\infty_{loc}(\R)$. 
Then one can use direction calculations (see, e.g., the proof of Proposition 2.5 in \cite{JLX}) to verify that $g$ is H\"older continuous. Finally, $h_2$ is smooth in $[-1,1]$, and one can use bootstrap arguments for $h_1$ (via difference quotients) to obtain that $g$ is smooth.

If $\sigma=\frac 12$ and $n\ge 2$, then $K(t)\le C|t|^{\frac 14}$, and therefore, it follows from the same reason that $g\in L^\infty_{loc}(\R)$, and therefore, is smooth.

Now, we will show that constant functions are not maximizers of $J_T$ if $T$ is sufficiently large. Note that for any nonzero constant $C$, $$J_T[C]=J_T[1]=\frac{\int_0^T \int_0^T K_T(t-s)\,\ud t\ud s}{T^{\frac{n+2\sigma}{n}}}.$$ 
For every $s\in [0,T]$,
\begin{align*}
\int_0^T K_T(t-s)\,\ud t =\int_{\R}K(t)\,\ud t<C.
\end{align*}
Hence
\[
J_T[1]\le CT^{-\frac{2\sigma}{n}}.
\]

On the other hand, let $\eta$ be a nonnegative smooth cut-off function such that $\eta\equiv 1$ on $[-1,1]$ and is identically zero outside of $(-2,2)$. Then
\[
J_T[\eta]\ge 2^{-\frac{n+2\sigma}{n}}\int_0^1\int_0^1K(t-s)\,\ud t\ud s>CT^{-\frac{2\sigma}{n}}\ge J_T[1]
\]
if $T\ge T^*$ for some sufficiently large $T^*$. Therefore, when $T\ge T^*$,  constant functions are not maximizers.

We complete the proof.
\end{proof}

Now let us deal with the case $n=1$.  We use subcritical approximations. Let $p>\frac{1-2\sigma}{1+2\sigma}$, and consider the subcritical problem: 
\[
J_{T,p}= \sup_{f\in \mathcal{A}_{T,p}\setminus \{0\}} J_{T,p}[f],
\]
where
\[
J_{T,p}[f]= \frac{\int_0^T \int_0^T K_T(t-s) f(t)f(s)\,\ud t\ud s}{(\int_0^T |f(t)|^{p+1}\ud t)^{\frac{2}{p+1}}},
\]
$K_T$ is as in \eqref{eq:KT}, and
\[
\mathcal{A}_{T,p}= \{f: f\in L^{p+1}_{loc}(\R), ~ f(t)=f(t+T)\quad  \mbox{for }t\in \R \}.
\]
The same proof of Proposition \ref{prop:fowler} leads to
\begin{prop} \label{prop:fowler2} 
Suppose $n=1$ and $p>\frac{1-2\sigma}{1+2\sigma}$. Then for every $T>0$, $J_{T,p}$ is achieved by a smooth positive function. 
\end{prop}

We will deal the critical case in 1-D, and show that Proposition \ref{prop:fowler} also holds for $n=1$.

\begin{prop} \label{prop:fowler3} 
Suppose $n=1$ and $\sigma\in (0,1/2)$. Then for every $T>0$, the $J_{T}$ defined in \eqref{eq:JT} is achieved by a smooth positive function. Moreover, there exists $T^*>0$ such that if  $T>T^*$, then the maximizers of $J_{T}$ are not constants. Consequently, there exists a family of non-constant periodic solutions to \eqref{eq:periodicsolution}.
\end{prop}
\begin{proof}
Let $p_i>\frac{1-2\sigma}{1+2\sigma}$ be such that $p_i\to \frac{1-2\sigma}{1+2\sigma}$ as $i\to\infty$. For brevity, we denote $J_{T,p_i}$ as $J_{T,i}$, and $J_{T,p_i}[f]$ as $J_{T,i}[f]$. 

We first have that
\be\label{eq:liminfbigger}
\liminf_{p_i\to\frac{1-2\sigma}{1+2\sigma}}J_{T,i}\ge J_{T}.
\ee
Indeed, for any $\va>0$, we can choose $0\le v\in L^\infty(0,T)$ such that 
\[
\|v\|_{L^{\frac{2}{1-2\sigma}}(0,T)}=1\quad\mbox{and}\quad\int_0^T \int_0^T K_T(t-s) v(t)v(s)\,\ud t\ud s>J_T-\va.
\]
Let $v_p=v/\|v\|_{L^{p+1}(0,T)}$. Then 
\[
\liminf_{p_i\to\frac{1-2\sigma}{1+2\sigma}}J_{T,i}\ge \liminf_{p_i\to\frac{1-2\sigma}{1+2\sigma}}J_{T,i}[v_{p_i}]\ge\liminf_{p_i\to\frac{1-2\sigma}{1+2\sigma}}\|v\|_{L^{p_i+1}(0,T)}^{-2}(J_T-\va)=J_T-\va.
\]
This proves \eqref{eq:liminfbigger} by sending $\va\to 0$. Therefore, one can assume that $J_{T,i}\to\Lda\ge J_{T}$.

Let $f_i$ be a maximizer of $J_{T,i}$ obtained in Proposition \ref{prop:fowler2}, which satisfies $\|f_i\|_{L^{p_i+1}(0,T)}=1$ and 
\[
J_{T,i}f_i(t)^{p_i}= \int_{0}^{T}K_T(t-s) f_i(s)\,\ud s=\int_{-\infty}^{\infty}K(t-s) f_i(s)\,\ud s.
\]
For convenience, we denote $q_i=\frac{1}{p_i}<\frac{1+2\sigma}{1-2\sigma}$ and $g_i(t)=f_i(t)^{p_i}$. Then
\[
J_{T,i}g_i(t)= \int_{-\infty}^{\infty}K(t-s) g_i(s)^{q_i}\,\ud s.
\]
We claim that $\{g_i\}$ is uniformly bounded on $[0,T]$. Then $\{g_i\}$ is uniformly bounded in $C^k$ norms, and will converge to a smooth function maximizing $J_T$. The proof of that constants are not maximizers of $J_T$ for large $T$ follows from the same proof as in that of Proposition \ref{prop:fowler}.

Our proof of the above claim is based on blow up analysis, which is similar to the proof of Proposition 2.11 in \cite{JLX}. Suppose not, that $m_i:=g_i(t_i)=\max_{[0,T]}f_i\to\infty$. Then subject to a subsequence, $t_i\to \bar t\in [0,T]$. Since the equation of $g_i$ is translation invariant, we may assume that $\bar t=0$. Let
\[
\tilde g_i(t)=m_i^{-1}g_i\left(m_i^{\frac{1-q_i}{2\sigma}}t+t_i\right).
\]
Then
\[
J_{T,i}\tilde g_i(t)= \int_{-\infty}^{\infty}K_i(t-s) \tilde g_i(s)^{q_i}\,\ud s.
\]
where
\begin{align*}
K_i(t)&= \frac{1}{|\widetilde m_i \sinh(\widetilde m_i^{-1}t )|^{1-2\sigma}}+\frac{1}{(\widetilde m_i \cosh(\widetilde m_i^{-1}t ))^{1-2\sigma}}\quad\mbox{with}\quad \widetilde m_i =2 m_i^{\frac{q_i-1}{2\sigma}}\\
& =: K_{i1}(t) +K_{i2}(t).
\end{align*}
Since $\tilde g_i$ is uniformly bounded by $1$, we have that $\|\tilde g_i\|_{C^\alpha(-R,R)}\le C(R)$, and thus, $\tilde g_i\to \tilde g$ in $C^{\alpha'}_{loc}(\R)$.  Since $\tilde g_i(0)=1$, we know that $\tilde g(0)=1$. Write
\[
J_{T,i}\tilde g_i(t)= \int_{-R}^{R}K_{i1}(t-s) \tilde g_i(s)^{q_i}\,\ud s +h_{i1}(R,t)+h_{i2}(t),
\]
where $R>0$ (large),
\[
h_{i1}(R,t)= \int_{|s|\ge R}K_{ij}(t-s) \tilde g_i(s)^{q_i}\,\ud s,\quad  h_{i2}(t)= \int_{\R}K_{i2}(t-s) \tilde g_i(s)^{q_i}\,\ud s,
\]
Since
\[
|K_{i1}'(t)|\le CK_{i1}(t)\quad\mbox{for all }t\ge 1,\quad\mbox{and}\quad |K_{i2}'(t)|\le Cm_i^{-1}K_{i2}(t)\quad\mbox{for all }t\in\R,
\]
we have that $|h_{i1}'(R,t)|\le C$  for $|t|\le R-1$, and $|h_{i2}'(t)|\le C\widetilde m_i^{-1}$ for all $t\in\R$. Thus, passing to a subsequence, $h_{i1}(R,t)\to h_1(R,t)$ and $h_{i2}'(t)\to c_2$ uniformly in $[-R+1,R-1]$, where $c_2$ is a nonnegative constant. Hence, sending $i\to\infty$, we obtain
\be\label{eq:blowupg}
\Lda \tilde g(t)= \int_{-R}^{R}\frac{\tilde g(s)^{\frac{1+2\sigma}{1-2\sigma}}}{|t-s|^{1-2\sigma}}\,\ud s +h_1(R,t)+c_2.
\ee
Arbitrarily fix $\va>0$. Choose $\delta>0$ small such that for $|t|< R/2$, if $R<|s|<\delta \widetilde m_i$, then 
\begin{align*}
\frac{1-\va}{1+\va} \frac{(R-|t|)^{2\sigma-1}}{R^{2\sigma-1}} 
\le\frac{1-\va}{1+\va} \frac{|s-t|^{2\sigma-1}}{|s|^{2\sigma-1}} &\le \frac{K_{i1}(t-s)}{K_{i1}(s)}\\
&\le \frac{1+\va}{1-\va}\frac{|s-t|^{2\sigma-1}}{|s|^{2\sigma-1}} \le \frac{1+\va}{1-\va} \frac{(R+|t|)^{2\sigma-1}}{R^{2\sigma-1}}. 
\end{align*}
If $|s|\ge\delta \widetilde m_i$, then it is elementary to check that
\begin{align*}
1-C(\delta)\frac{R}{m_i} \le \frac{K_{i1}(t-s)}{K_{i1}(s)}\le 1+C(\delta)\frac{R}{m_i}. \end{align*}
Therefore, we have for $|t|<R/2$.
\[
\frac{1-\va}{1+\va} \frac{(R-|t|)^{2\sigma-1}}{R^{2\sigma-1}}  \le \frac{h_1(R,t)}{h_1(R,0)}\le  \frac{1+\va}{1-\va} \frac{(R+|t|)^{2\sigma-1}}{R^{2\sigma-1}}. 
\]
Since $h_1(R,t)$ is non-increasing, by sending $R\to\infty$ and then $\va\to 0$, we obtain that 
\[
\lim_{R\to\infty} h_1(R,t)=\lim_{R\to\infty} h_1(R,0)=c_1\ge 0.
\]
Therefore, from \eqref{eq:blowupg} we have
\[
\Lda \tilde g(t)= \int_{\R}\frac{\tilde g(s)^{\frac{1+2\sigma}{1-2\sigma}}}{|t-s|^{1-2\sigma}}\,\ud s +c_1+c_2.
\]
Since $g$ is locally bounded, we have $c_1+c_2=0$, and thus,
\[
\Lda \tilde g(t)= \int_{\R}\frac{\tilde g(s)^{\frac{1+2\sigma}{1-2\sigma}}}{|t-s|^{1-2\sigma}}\,\ud s.
\]
The solutions of the above equation have been classified in \cite{CLO} and \cite{Li04}. 

Then
\begin{align*}
\Lda \int_{\R}\tilde g(t)^{\frac{2}{1-2\sigma}}=\int_{\R}\int_{\R}\frac{\tilde g(s)^{\frac{1+2\sigma}{1-2\sigma}}\tilde g(t)^{\frac{1+2\sigma}{1-2\sigma}}}{|t-s|^{1-2\sigma}}\,\ud s\,\ud s\le S(\sigma) \left(\int_{\R}\tilde g(t)^{\frac{2}{1-2\sigma}}\right)^{1+2\sigma},
\end{align*}
where $S(\sigma)$ is the sharp constant in the Hardy-Littlewood-Sobolev inequality \cite{Lieb83} in 1-D:
\be\label{eq:sharpHLS}
S(\sigma)=\sup_{f\in L^{\frac{2}{1+2\sigma}}(\R),\ f\neq 0}\frac{\|\int_{\R}|x-y|^{2\sigma-1}f(y)\,\ud y\|_{L^{\frac{2}{1-2\sigma}}(\R)}}{\|f\|_{L^{\frac{2}{1+2\sigma}}(\R)}}.
\ee
Since $\|f_i\|_{L^{p_i+1}(0,T)}=1$ and $f_i$ is periodic with period $T$, we have, for any fixed $R>0$,
\[
1=m_i^{\frac{1-2\sigma}{2\sigma}(\frac{1+2\sigma}{1-2\sigma}-q_i)}\int_{-Tm_i^{\frac{q_i-1}{2\sigma}}/2}^{Tm_i^{\frac{q_i-1}{2\sigma}}/2} \tilde g_i(t)^{q_i+1}\,\ud t \ge \int_{-R}^R  \tilde g_i(t)^{q_i+1}\,\ud t \to  \int_{-R}^R  \tilde g(t)^{\frac{2}{1-2\sigma}}\,\ud t. 
\]
Hence, $\int_{\R}  \tilde g(t)^{\frac{2}{1-2\sigma}}\,\ud t\le 1$, and thus
\be\label{eq:threasholdvalue}
J_{T}\le \Lda\le S(\sigma).
\ee

Now we are going to show that $J_{T}>S(\sigma)$, which will reach a contradiction. Let
\[
u_\lda(t)=\left(\frac{\lda}{\lda^2+|t-T/2|^2}\right)^{\frac{1+2\sigma}{2}},
\]
which are maximizers of $S(\sigma)$ in \eqref{eq:sharpHLS} for all $\lda>0$; see Lieb \cite{Lieb83}. Moreover,
\[
\int_{\R} u_\lda(t)^{\frac{2}{1+2\sigma}}\,\ud t=\int_{\R}\frac{1}{1+t^2}\,\ud t=\pi,
\]
and
\[
S(\sigma)\pi^{2\sigma} u_\lda^{\frac{1-2\sigma}{1+2\sigma}}(t)= \int_{\R}\frac{u_\lda(s)}{|t-s|^{1-2\sigma}}\,\ud s.
\]

Let $f_\lda=u_\lda$ on $[0,T]$, and extend it to be a periodic function on $\R$ with period $T$. Then 
\begin{align*}
\int_0^T f_\lda(t)^{\frac{2}{1+2\sigma}}\,\ud t\le \int_{\R} u_\lda(t)^{\frac{2}{1+2\sigma}}\,\ud t=\pi.
\end{align*}
For $|t|<1$, we have $|\sinh t|\le |t|+|t|^3$, and thus,
\[
\frac{1}{|2\sinh(\frac t2)|^{1-2\sigma}}\ge \frac{1}{|t|^{1-2\sigma}}\frac{1}{|1+|t|^2/4|^{1-2\sigma}}\ge \frac{1}{|t|^{1-2\sigma}}\left(1-\frac{|t|^2}{4}\right)=\frac{1}{|t|^{1-2\sigma}}-\frac{|t|^{1+2\sigma}}{4}.
\]
Choosing $\delta>0$ universally small such that for all $|t|\le 4\delta$, 
\[
-\frac{|t|^{1+2\sigma}}{4}+\frac{1}{|2\cosh(\frac t2)|^{1-2\sigma}}\ge \frac{1}{4^{1-2\sigma}}.
\]
Then for $|t|\le 4\delta$, we have
\[
K(t)=\frac{1}{|2\sinh(\frac t2)|^{1-2\sigma}}+\frac{1}{|2\cosh(\frac t2)|^{1-2\sigma}}\ge  \frac{1}{|t|^{1-2\sigma}}+ \frac{1}{4^{1-2\sigma}}.
\]
Then
\begin{align*}
&\int_0^T \int_0^T  K_T(t-s) f_\lda(s)f_\lda(t)\,\ud t\ud s\\
&\ge \int_{-\delta}^{\delta} u_\lda(t+T/2)\int_{-3\delta}^{3\delta}K(t-s) u_\lda(s+T/2)\,\ud s\ud t\\
&\ge \int_{-\delta}^{\delta} u_\lda(t+T/2)\int_{-3\delta}^{3\delta}\frac{u_\lda(s+T/2)}{|t-s|^{1-2\sigma}}\,\ud s\ud t+ \frac{1}{4^{1-2\sigma}} \left(\int_{-\delta}^{\delta} u_\lda(t+T/2)\,\ud t\right)^2\\
&\ge \int_{-\delta}^{\delta} u_\lda(t+T/2)\int_{\R}\frac{u_\lda(s+T/2)}{|t-s|^{1-2\sigma}}\,\ud s\ud t-C\lda+A \lda^{1-2\sigma}\\
&= S(\sigma)\pi^{2\sigma} \int_{-\delta}^{\delta} u_\lda(t+T/2)^{\frac{2}{1+2\sigma}}\ud t-C\lda+A \lda^{1-2\sigma}\\
&= S(\sigma)\pi^{1+2\sigma} -C\lda+A \lda^{1-2\sigma},
\end{align*}
where $A$ and $C$ are universal positive constants. Therefore,
\[
J_T\ge J_{T}[f_\lda]\ge \frac{\int_0^T \int_0^T  K_T(t-s) f_\lda(s)f_\lda(t)\,\ud t\ud s}{(\int_0^T f_\lda(t)^{\frac{2}{1+2\sigma}}\,\ud t)^{1+2\sigma}}\ge \frac{S(\sigma)\pi^{1+2\sigma} -C\lda+A \lda^{1-2\sigma}}{\pi^{1+2\sigma}}>S(\sigma)
\]
if $\lda>0$ is sufficiently small.
 
This finishes the proof of this proposition.
\end{proof}

\begin{proof}[Proof of Theorem \ref{thm:B}] It follows immediately from Propositions \ref{prop:fowler} and \ref{prop:fowler3}.
\end{proof}

\section{A Harnack type inequality of Schoen}\label{sec:harnack}

\begin{proof}[Proof of Theorem \ref{thm:C}]
The proof is very similar to that of Proposition \ref{prop:upbound}. The only visible difference is that we do not have \eqref{eq:wj-l-1} here. We still use the method of moving spheres.

By the tranformations $u_R(x)=R^{\frac{n-2\sigma}{2}} u(Rx)$ and $h_R(x)=R^{\frac{n-2\sigma}{2}} h(Rx)$, and observing
\[
\|h_R\|_{L^\infty(B_{5/2})}+\|\nabla \ln h_R\|_{L^\infty(B_{5/2})}\le A,
\]
we only need to prove \eqref{eq:thmcharnack} for $R=1$. We extend $u$ to be identically $0$ outside $B_3$, then we have
\be \label{eq:harnack1}
u(x)= \int_{\R^n} \frac{u(y)^{\frac{n+2\sigma}{n-2\sigma}}}{|x-y|^{n-2\sigma}}\,\ud y +h(y) \quad \mbox{for } x\in B_3.
\ee 
  
Suppose the contrary that there exists a sequence of solutions $u_j$ of \eqref{eq:harnack1} with $h_j$ such that
\[
u_j(x_j)\min_{\overline B_2} u_j>j\quad \mbox{as } j\to\infty,
\]
where $u_j(x_j)=\max_{\overline B_1}u_j$.

Consider
\[
v_j(x):=\left(1/2-|x-x_j|\right)^{\frac{n-2\sigma}{2}} u_j(x),\quad |x-x_j|\leq 1/2.
\]
Let $|\bar x_j-x_j|<1/2$ satisfy
\[
v_j(\bar x_j)=\max_{|x-x_j|\leq 1/2}v_j(x),
\]
and let
\[
2\mu_j:=1/2-|\bar x_j-x_j|.
\]
Then
\[
0<2\mu_j\leq 1/2\quad\mbox{and}\quad 1/2-|x-x_j|\ge\mu_j \quad \forall ~ |x-\bar x_j|\leq \mu_j.
\]
By the definition of $v_j$, we have
\be\label{eq:cl3-5}
(2\mu_j)^{\frac{n-2\sigma}{2}}u_j(\bar x_j)=v_j(\bar x_j)\ge v_j(x)\ge (\mu_j)^{\frac{n-2\sigma}{2}}u_j(x)\quad \forall ~ |x-\bar x_j|\leq \mu_j.
\ee
Thus, we have
\[
2^{\frac{n-2\sigma}{2}}u_j(\bar x_j)\ge u_j(x)\quad \forall ~ |x-\bar x_j|\leq \mu_j.
\]
We also have
\be\label{eq:cl4-5}
\begin{split}
(2\mu_j)^{\frac{n-2\sigma}{2}}u_j(\bar x_j)=v_j(\bar x_j)\ge v_j(x_j)&=2^{\frac{2\sigma-n}{2}} u_j(x_j)\\
&\ge 2^{\frac{2\sigma-n}{2}}\sqrt{u_j(x_j)\min_{\overline B_2} u_j}\\
&\ge 2^{\frac{2\sigma-n}{2}}\sqrt{j} \to \infty.
\end{split}
\ee
Let
\[
w_j(y)=\frac{1}{u_j(\bar x_j)} u_j\left(\bar x_j +\frac{y}{u_j(\bar x_j)^{\frac{2}{n-2\sigma}}}\right),\quad  \tilde h_j(y)=\frac{1}{u_j(\bar x_j)} h_j\left(\bar x_j +\frac{y}{u_j(\bar x_j)^{\frac{2}{n-2\sigma}}}\right)  \quad \mbox{in }\om_j,
\]
where
\[
\om_j=\Big \{y\in \R^n:\bar x_j +\frac{y}{u_j(\bar x_j)^{\frac{2}{n-2\sigma}}} \in B_2\Big\}.
\]
Then
\[
 w_j(y) =\int_{\R^n} \frac{w_j(z)^{\frac{n+2\sigma}{n-2\sigma}}}{|y-z|^{n-2\sigma}}\,\ud z +\tilde h_j(y)  \quad \mbox{for }y\in  \om_j
\]
and $w_j(0)=1$. Moreover, it follows from \eqref{eq:cl3-5} and \eqref{eq:cl4-5} that
\[
\|\tilde h_j\|_{C^0(\om_j)}\to 0, \quad  w_j(y)\leq 2^{\frac{n-2\sigma}{2}} \quad\mbox{in } B_{R_j},
\]
where \[R_j:=\mu_j u(\bar x_j)^{\frac{2}{n-2\sigma}}\to \infty \mbox{ as } j\to \infty.\] By the regularity results in Section 2.1 of \cite{JLX}, and the proof of Proposition 2.9 in \cite{JLX}, there exists $w>0$ such that
\[
w_j\to w \quad \mbox{in } C^{\al}_{loc}(\R^n)
\]
for some $\alpha>0$, and $w$ satisfies
\[
 w(y) =\int_{\R^n} \frac{w(z)^{\frac{n+2\sigma}{n-2\sigma}}}{|y-z|^{n-2\sigma}}\,\ud z   \quad \mbox{for }y \in \R^n.
\] 
Since $w(0)=1$, by the classification results in \cite{CLO} or \cite{Li04}, we have
\be \label{eq:cl5harnack}
w(y)=\left(\frac{1+\mu^2|y_0|^2}{1+\mu^2|y-y_0|^2}\right)^{\frac{n-2\sigma}{2}}
\ee
for some $\mu>0$ and some $y_0\in\R^n$. 

On the other hand, we are going to show that, for every $\lda>0$
\be\label{eq:aimharnack1}
w_{\lda}(y)\leq w(y)\quad \forall ~ |y|\ge\lda,
\ee
which will have a contradiction to \eqref{eq:cl5harnack}.

Let us arbitrarily fix $\lda_0>0$. Then for all $j$ large, we have $0<\lda_0<\frac{R_j}{10}$. Let
\[
\Sigma_j:=B_{\Gamma_j} \subset \subset \om_j,\quad\mbox{where }\Gamma_j=\frac{1}{4}u_j(\bar x_j)^\frac{2}{n-2\sigma}.
\]
We will show that for all sufficiently large $j$,
\be\label{eq:aim11harnack}
(w_j)_{\lda_0}(y)\leq w_j(y)\quad \forall ~ |y|\ge\lda_0,\ y\in \Sigma_j.
\ee
Then \eqref{eq:aimharnack1} follows from \eqref{eq:aim11harnack} by sending $j\to \infty$.

By Lemma \ref{lem:fix_error}, there exist $\bar r>0$  such that $h_j$ satisfies \eqref{eq:fix-er} for all $j$. Let $j$ be large such that  $\lda_0 u_j(\bar x_j)^{-\frac{1}{n-2\sigma}} <\bar r$. Thus, \eqref{eq:moving-f} holds for $\tilde h_j$, that is, for every $0<\lda\le \lda_0$
\[
(\tilde h_j)_{\lda}(y) \le  \tilde h_j(y) \quad \mbox{for }y\in \Sigma_j \backslash B_{\lda}.
\]

As in the proof of Proposition \ref{prop:upbound}, we first know that
\[
\bar \lda:=\sup \{0<\mu\le \lda_0\ |\ (w_j)_{\lda}(y)\leq w_{j}(y),\ \forall~|y-x_0|\geq \lda, ~y\in \Sigma_j,~\forall~ 0<\lda <\mu\},
\]
 is well defined.
 
 We are going to show that $\bar\lda=\lda_0$ for all sufficiently large $j$. In this step, there will be some small changes from the proof of Proposition \ref{prop:upbound}, and we explain it in the below.

We start from \eqref{eq:movingsphere} with $J$ defined in \eqref{eq:definitionofJ}. 

For $z\in \R^n\setminus \Sigma_j$ and $\bar \lda \le \lda \le \bar \lda+1$, we have $|z| \ge \frac14u_j(\bar x_j)^{\frac{2}{n-2\sigma}}$ and thus
\[
(w_j)_{\lda}(z) \le \left(\frac{\lda }{|z|}\right)^{n-2\sigma} \max_{B_{\bar \lda+1}} w_j \le C u_j(\bar x_j)^{-2}.
\]
By the definition of $w_j$, we have
\[
w_j(y)\ge \frac{\min_{\overline B_2}u_j}{u_j(\bar x_j)}\ge \frac{j}{u_j(\bar x_j)u_j(x_j)}\ge2^{\frac{2\sigma-n}{2}}\frac{j}{u_j(\bar x_j)^2} \quad \mbox{for all }y\in\om_j.
\]
It follows that for large $j$,
\[
 w_j(z)^{\frac{n+2\sigma}{n-2\sigma}}-(w_j)_{ \lda}(z)^{\frac{n+2\sigma}{n-2\sigma}} \ge ju_j(\bar x_j)^{-\frac{2n+4\sigma}{n-2\sigma}}\quad\mbox{in }\Omega_j\setminus\Sigma_j.
\]
Then, as in proving \eqref{eq:lower-J}, we can show that
\begin{align}
J(\lda, w_j,y) &\ge ju_j(\bar x_j)^{-\frac{2n+4\sigma}{n-2\sigma}}\int_{\om_j \setminus \Sigma_j } K(0,\bar \lda; y,z) \,\ud z- C \int_{\om_j^c} K(0,\bar \lda; y,z) \left(\frac{\lda }{|z|}\right)^{n+2\sigma}   \,\ud z  \nonumber  \\&
\ge \begin{cases}
\frac{j}{2} (|y|-\lda) u(\bar x_j)^{-\frac{2n}{n-2\sigma}},& \quad \mbox{if }\lda\le |y|\le \bar \lda +1,\\
\frac{j}{2}u(\bar x_j)^{-\frac{2n}{n-2\sigma}}, & \quad \mbox{if } |y|>\bar \lda +1, y\in\Sigma_j.
\end{cases}
\label{eq:lower-J-harnack}
\end{align}
Consequently, by \eqref{eq:movingsphere} and \eqref{eq:lower-J-harnack}, there exists $\va_1\in (0,1/2)$ (which depends on $j$) such that
\[
w_j(y)-(w_j)_{\bar \lda}(y) \ge \frac{\va_1}{|y|^{n-2\sigma}} \quad \forall~|y|\ge \bar \lda+1,\ y\in\Sigma_j.
\]
The rest proof of $\bar \lda=\lda_0$ is identical to that of Proposition \ref{prop:upbound}, and we do not repeat here.
\end{proof}

\begin{appendices}

\section{Pohozaev identities}\label{appendix:pohozaev}

The following is the Pohozaev identity needed for the case that $\sigma$ is an integer.
\begin{prop}[Pohozaev type identity, Proposition 3.3 in \cite{GPY}] \label{prop:GPY} If $\sigma$ is an integer, we have, for any $u,v\in C^{2\sigma}(B_2\setminus \{0\})$,
\begin{align*}
&\int_{B_R\setminus B_r } (-\Delta)^\sigma u\  \langle x,\nabla v\rangle+ (-\Delta)^\sigma v \  \langle x,\nabla u\rangle\\& = \int_{\pa B_R} g_\sigma(u,v)\,\ud s-\int_{\pa B_r} g_\sigma(u,v)\,\ud s  -\frac{n-2\sigma}{2} \int_{B_R\setminus B_r } \Big((-\Delta)^\sigma u v+ (-\Delta)^\sigma v u\Big),
\end{align*}
where $0<r<R<1$, $g_\sigma(u,v)$  has the following form
\begin{equation}\label{eq:gpohozaev}
g_\sigma(u,v)=\sum_{j=1}^{2\sigma-1}\bar l_j (x, \nabla^j u, \nabla^{2\sigma-j} v) + \sum_{j=0}^{2\sigma-1} \tilde l_j (\nabla^j u, \nabla^{2\sigma-j-1} v)
\end{equation}
and $\bar l_j (x, \nabla^j u, \nabla^{2\sigma-j} v)$ and $ \tilde l_j (\nabla^j u, \nabla^{2\sigma-j-1} v)$ are linear in each component. Moreover
\begin{equation}\label{eq:pohozaevleft}
\int_{\pa B_R}\tilde l_{2\sigma-1}(\nabla^{2\sigma-1} u, v)=(-1)^\sigma\frac{n-2\sigma}{2} \int_{\pa B_R} \frac{\pa \Delta^{\sigma-1} u}{\pa\nu} v,
\end{equation}
where $\nu$ is the unit outer normal of $B_R$.
\end{prop}

Suppose $\sigma$ is an integer. Let
\begin{equation}\label{eq:pohozaevinteger}
P(u,r)= \int_{\pa B_r} g_\sigma(u,u) \,\ud s - \frac{n-2\sigma}{n} c(n,\sigma)r \int_{\pa B_r} u^{\frac{2n}{n-2\sigma}}\,\ud s,
\end{equation}
where $g_\sigma$ is given in \eqref{eq:gpohozaev}.

\begin{prop} \label{prop:pohozaev}
Suppose $\sigma$ is an integer, and $u\in C^{2\sigma}(B_2\setminus\{0\})$ is a positive solution of 
\[
(-\Delta)^\sigma u= c(n,\sigma) u^{\frac{n+2\sigma}{n-2\sigma}} \quad \mbox{in }B_2\setminus\{0\}.
\]
Let $P(u,r)$ be as in \eqref{eq:pohozaevinteger}. Then $P(u,r)$ is  independent of $r$.
\end{prop}

\begin{proof}   
It follows from Proposition \ref{prop:GPY}, the equation of $u$ and integration by parts.
\end{proof}

If $\sigma$ is not an integer, then we will use an extension formulation for $(-\Delta)^\sigma$ (see \cite{CS2,CG,Y}), so that we will work in $\R^{n+1}$. We use capital letters, such as $X=(x,t)\in\R^n\times\R$, to denote points in $\R^{n+1}$. Let $m\ge 0$ be an integer, 
\[
\Delta_b:= \Delta_{x,t} +\frac{b}{t} \pa_t=t^{-b} \mathrm{div}(t^{b} \nabla ),
\]
and $\Delta_b^m U=\Delta_b(\Delta_b^{m-1} U)$. We use the convention that $\Delta_b^0 U=U$. We will derive a Pohozaev-type identity for $\Delta_b^{m+1} U$. A Pohozaev-type identity for $\Delta_b^{2} U$ was obtained and used in \cite{FF}.

\begin{prop}[Pohozaev type identity] \label{prop:GPYsigma}  Let $E \subset \R^{n+1}$ be a bounded piecewise-smooth domain and $U\in C^{2m+2}(\overline E)$. Then
\begin{align*}
\int_{E}\mathrm{div}(t^{b}\nabla \Delta_b^m U) \langle X,\nabla U\rangle 
&=-\frac{n+b-1-2m}{2} \int_E \mathrm{div}(t^{b} \nabla \Delta_b^{m}  U) U\\
& \quad + \frac{n+b-1-2m}{2}  \int_{\pa E} t^b \pa_\nu \Delta_b^m U U \\
&\quad +  \int_{\pa E} t^b \pa_\nu \Delta_b^m U \langle X,\nabla U\rangle +\int_{\pa E}Q_m(U),
\end{align*}
where $\nu$ is the unit outer normal of $E$, 
\[
 \begin{split}
 Q_0(U)&=-\frac 12 \int_{\pa E} t^b \nabla U\nabla U X^k\nu_k  \\
Q_1(U)&=\int_{\pa E} \Big( t^b \pa_\nu U X^{k}\pa_k \Delta_b U+ t^b  \Delta_b U  \pa_\nu U - t^b \nabla U\nabla \Delta_b U X^k\nu_k  -t^b \Delta_b U \nu_lX^k \pa_{kl}U \Big)\\
Q_m(U)&=\int_{\pa E} Q_{m-2}(\Delta_bU)  \\
&\quad+ \frac{(n-1+b-2(m-2))}{2}\int_{\pa E} t^b \Delta_bU\pa_\nu \Delta_b^{m-1}U+\int_{\pa E} t^b \pa_\nu \Delta_b^{m-1} U X^k\pa_k \Delta_bU \\
&\quad  -\frac{(n-1+b-2(m-1))}{2}\int_{\pa E} t^b \pa_\nu U \Delta_b^mU-\int_{\pa E} t^b \Delta_b^m U \nu_lX^k \pa_{kl}U
\end{split}
\]
for $m\ge 2$.

\end{prop}

\begin{proof} 
We use $X=(X^1,\cdots,X^n,X^{n+1})$, and $X^k\pa_k U=\langle X,\nabla U\rangle=\sum_{k=1}^{n+1}X^k\pa_{X^k} U$.

We first prove it for $m=0$. Let $V$ satisfy the same conditions as $U$. By integration by parts,
\begin{align*}
&\int_{E}\mathrm{div}(t^{b}\nabla U) X^k\pa_k V=-\int_{E} t^b\nabla U \nabla(X^k \pa_k V) +\int_{\pa E} t^b \pa_\nu U X^{k}\pa_k V\\&
=-\int_{E} t^b\nabla U (\nabla V+X^k \pa_k \nabla V)  +\int_{\pa E} t^b \pa_\nu U X^{k}\pa_k V \\&
=\int_E \pa_k(X^k t^b \nabla U)\nabla  V -\int_{E} t^b \nabla U\nabla V - \int_{\pa E }t^b \nabla U\nabla V X^k\nu_k  +\int_{\pa E} t^b \pa_\nu U X^{k}\pa_k V \\&
=\int_E  t^b \nabla (X^k\pa_k U)\nabla  V +(n-1+b)\int_{E} t^b \nabla U\nabla V - \int_{\pa E }t^b \nabla U\nabla V X^k\nu_k  \\
&\quad+\int_{\pa E} t^b \pa_\nu U X^{k}\pa_k V.
\end{align*}
Integrating by parts again, it follows that
\begin{align*}
&\int_{E}\mathrm{div}(t^{b}\nabla U) X^k\pa_k V  +\int_{E}\mathrm{div}(t^{b}\nabla V) X^k\pa_k U   \\&
=-\frac{(n-1+b)}{2}\int_{E} U\mathrm{div}(t^b\nabla V) + V\mathrm{div}(t^b\nabla U) +\frac{(n-1+b)}{2} \int_{\pa E} t^b U\pa_\nu V+t^b \pa_\nu UV  \\&
\quad  +\int_{\pa E} \Big(t^b X^k\pa_k U \pa_\nu V + t^b X^{k}\pa_k V\pa_\nu U - t^b \nabla U\nabla V X^k\nu_k  \Big).
\end{align*}
This proves the case when $m=0$ by setting $V=U$. 

For $m\ge 1$, we have
\begin{equation}
\label{eq:poho-1}
\begin{split}
&\int_{E}\mathrm{div}(t^{b}\nabla \Delta_b^m U) X^k\pa_k V  \\&=-\int_{E} t^b\nabla \Delta_b^m U (\nabla V+X^k \pa_k \nabla V)  +\int_{\pa E} t^b \pa_\nu \Delta_b^m U X^{k}\pa_k V\\&
= \int_E \mathrm{div}(t^{b} \nabla \Delta_b^{m-1} U) X^k \pa_k \Delta_b V +\Delta^{m}_b U\mathrm{div}(t^b  \nabla V) - t^b\nabla \Delta_b^m U \nabla V\\&
\quad +\int_{\pa E} t^b \pa_\nu \Delta_b^m U X^{k}\pa_k V -\int_{\pa E} t^b \Delta_b^m U \nu_lX^k \pa_{kl}V\\&
= \int_E \mathrm{div}(t^{b} \nabla \Delta_b^{m-1} U) X^k \pa_k \Delta_b V +2\mathrm{div}(t^{b} \nabla \Delta_b^{m} U)V  \\&\quad +\int_{\pa E} t^b \pa_\nu \Delta_b^m U X^{k}\pa_k V -\int_{\pa E} t^b \Delta_b^m U \nu_lX^k \pa_{kl}V \\&
\quad +\int_{\pa E} t^b  \Delta_b^m U  \pa_\nu V-2 \int_{\pa E}  t^b  \pa_\nu  \Delta_b^m U V,
\end{split}
\end{equation}
where we used
\begin{align*}
\mathrm{div} (t^b X^k\pa_k \nabla V)&= X^k \pa_k \mathrm{div}(t^b  \nabla V) +(1-b)\mathrm{div}(t^b  \nabla V)\\&
=t^b X^k \pa_k \Delta_b V +\mathrm{div}(t^b  \nabla V).
\end{align*}
If $m=1$, by what we have proved for $m=0$, it follows that
\begin{align*}
& \int_E \mathrm{div}(t^{b} \nabla U) X^k \pa_k \Delta_b V +\mathrm{div}(t^{b} \nabla \Delta_b V) X^k \pa_k U \\&
=-\frac{(n-1+b)}{2}\int_{E} U\mathrm{div}(t^b\nabla \Delta_b V) + \Delta_b V\mathrm{div}(t^b\nabla U)\\ &\quad+\frac{(n-1+b)}{2} \int_{\pa E} t^b U\pa_\nu \Delta_bV+t^b \pa_\nu U \Delta_b V  \\&
\quad  +\int_{\pa E} \Big( t^b X^k\pa_k U \pa_\nu \Delta_b V - t^b \nabla U\nabla \Delta_b V X^k\nu_k  + t^b \pa_\nu U X^{k}\pa_k \Delta_b V\Big).
\end{align*}
Note that
\begin{align}\label{eq:integration UV}
\int_E \Delta_b V\mathrm{div}(t^b\nabla U) =\int_E V \mathrm{div}(t^b\nabla \Delta_b U) +\int_{\pa E} t^b \pa_\nu V \Delta_bU- t^bV \pa_\nu \Delta_b U.
\end{align}
By setting $V=U$, we obtain
\begin{align*}
&\int_{E}\mathrm{div}(t^{b}\nabla \Delta_b U) X^k\pa_k U\\
& =-\frac{(n-1+b-2)}{2}\int_{E} U\mathrm{div}(t^b\nabla \Delta_b U) +\frac{(n-1+b-2)}{2} \int_{\pa E} t^b U\pa_\nu \Delta_bU  \\
&\quad+\int_{\pa E} t^b \pa_\nu \Delta_b U X^k\pa_k U  \\
&+\int_{\pa E} \Big( t^b \pa_\nu U X^{k}\pa_k \Delta_b U+ t^b  \Delta_b U  \pa_\nu U - t^b \nabla U\nabla \Delta_b U X^k\nu_k  -t^b \Delta_b U \nu_lX^k \pa_{kl}U \Big).
\end{align*}

Hence, we proved it for $m=1$. 

Suppose it holds up to  $m-1$ with $m\ge 2$. We are going to show it for $m$. We start from \eqref{eq:poho-1} and calculate the first term on the right-hand side of \eqref{eq:poho-1}:
\begin{align*}
&\int_E \mathrm{div}(t^{b} \nabla \Delta_b^{m-1} U) X^k \pa_k \Delta_b U\\
&=\int_E \mathrm{div}(t^{b} \nabla \Delta_b^{m-2} (\Delta_bU)) X^k \pa_k \Delta_b U  \\
& =\frac{(n-1+b-2(m-2))}{2}\left(-\int_{E} \Delta_bU\mathrm{div}(t^b\nabla \Delta_b^{m-1} U) +\int_{\pa E} t^b \Delta_bU\pa_\nu \Delta_b^{m-1}U\right)  \\
&\quad+\int_{\pa E} t^b \pa_\nu \Delta_b^{m-1} U X^k\pa_k \Delta_bU  +\int_{\pa E} Q_{m-2}(\Delta_bU).
\end{align*}
By \eqref{eq:integration UV}, we have
\begin{align*}
\int_E \Delta_b U\mathrm{div}(t^b\nabla \Delta_b^{m-1}U) &=\int_E U \mathrm{div}(t^b\nabla \Delta_b^m U) +\int_{\pa E} t^b \pa_\nu U \Delta_b^mU- t^bU \pa_\nu \Delta_b^m U.
\end{align*}
Combining with \eqref{eq:poho-1} with $U=V$, we have
\begin{align*}
&\int_{E}\mathrm{div}(t^{b}\nabla \Delta_b^m U) X^k\pa_k U  \\
& =-\frac{(n-1+b-2m)}{2}\int_{E} \Delta_bU\mathrm{div}(t^b\nabla \Delta_b^{m-1} U)  +\frac{(n-1+b-2m)}{2}\int_{\pa E}  t^bU \pa_\nu \Delta_b^m U\\
&\quad+\int_{\pa E} t^b \pa_\nu \Delta_b^m U X^{k}\pa_k U +\int_{\pa E} Q_{m-2}(\Delta_bU)  \\
&\quad+ \frac{(n-1+b-2(m-2))}{2}\int_{\pa E} t^b \Delta_bU\pa_\nu \Delta_b^{m-1}U+\int_{\pa E} t^b \pa_\nu \Delta_b^{m-1} U X^k\pa_k \Delta_bU \\
&\quad  -\frac{(n-1+b-2(m-1))}{2}\int_{\pa E} t^b \pa_\nu U \Delta_b^mU-\int_{\pa E} t^b \Delta_b^m U \nu_lX^k \pa_{kl}U 
\end{align*}
This completes the proof.
\end{proof}

\section{Estimates for Poisson extensions}\label{appendix:poisson}
This appendix is to show some estimates for Poisson extensions, which are needed for the proof of Proposition \ref{prop:zero} when $\sigma\in (0,\frac n2)$ is not an integer. 

Let
\[
\mathcal{P}_\sigma(x,t)= \beta(n,\sigma) \frac{t^{2\sigma}}{(|x|^2+t^2)^{\frac{n+2\sigma}{2}}}, \quad (x,t)\in \R^{n+1}_+,
\]
where $\beta(n,\sigma)$ is a normalization constant such that $ \beta(n,\sigma)  \int_{\R^n}(|x|^2+1)^{-\frac{n+2\sigma}{2}}\,\ud x=1$, and let
 \begin{equation}\label{eq:poissonextension}
 V(x,t)= \mathcal{P}_\sigma * v(x,t)= \beta(n,\sigma) \int_{\R^n} \frac{t^{2\sigma}}{(|x-y|^2+t^2)^{\frac{n+2\sigma}{2}}}v(y)\,\ud y.
 \end{equation}
Let $m=[\sigma]$, $b=1-2(\sigma-m) $ and
\[
\Delta_b:= \Delta_{x,t} +\frac{b}{t} \pa_t=t^{-b} \mathrm{div}(t^{b} \nabla ).
\]
If $v$ and its derivatives have fast decay at infinity (e.g., in certain fractional Sobolev spaces), then we have $V(x,0)=v(x)$ and
\begin{align}
\Delta_b^{m+1} V&=0 \quad \mbox{in }\R^{n+1}_+,\label{eq:extensionharmonic}\\
\lim_{t\to 0} t^{b} \pa_t \Delta_b^k V&=0 , \quad k=0,\dots, m-1,\label{eq:extensionboundaryzero}\\
(-1)^{m+1}\lim_{t\to 0} t^{b} \pa_t \Delta_b^m V&= N_{n,\sigma} (-\Delta)^\sigma v(x).\label{eq:extensionneumann}
\end{align}
where $N_{n,\sigma}$ is a positive constant depending only on $n, \sigma$. The equations \eqref{eq:extensionharmonic}, \eqref{eq:extensionboundaryzero} and \eqref{eq:extensionneumann}  were proved Caffarelli-Silvestre \cite{CS2} for $\sigma\in(0,1)$ and  Yang\cite{Y} for $\sigma\in (1,\frac n2)$. 


We will verify that the equations \eqref{eq:extensionharmonic}, \eqref{eq:extensionboundaryzero} and \eqref{eq:extensionneumann} also hold for our solution $v$ defined in \eqref{eq:definitionofv} in certain domains. Denote $\B_R$ as the ball in $\R^{n+1}$ with radius $R$ and center at the origin, $\B^+_R$ as the upper half ball $\B_R\cap \R^{n+1}_+$, $\pa' \B_R$ as the flat part of $\pa \B_R$ which is the ball $B_R$ in $\R^{n}$ and $\pa'' \B_R= \pa \B_R \cap \{t>0\}$.

\begin{prop}\label{lem:cut-off1}
Let $v\in W^{2m,1}_{loc}(\R^n) \cap C^{\infty}(B_2)$ be a positive function in $\R^n$, where $m=[\sigma]$. Suppose that 
\begin{align}\label{eq:cut-offassumption}
\int_{\R^n}\frac{|\nabla^kv(y)|}{1+|y|^{n+2\sigma-k}}\,\ud y<\infty\quad, k=0,1,\cdots, 2m. 
\end{align}
Let $V$ be as in \eqref{eq:poissonextension}. Then the left hand sides of \eqref{eq:extensionboundaryzero} and \eqref{eq:extensionneumann} are H\"older continuous functions in $\overline{\B_1^+}$, and both \eqref{eq:extensionboundaryzero} and \eqref{eq:extensionneumann} hold for all $x\in B_1$.
\end{prop}
\begin{proof}
Let $\eta$ be a nonnegative smooth cut-off function such that $\eta\equiv 1$ in $[-1,1]$ and is identically zero outside of $[-3/2,3/2]$. Let $(x,t)\in\B_1^+$. Then
\begin{align*}
 V(x,t)&= \beta(n,\sigma) \int_{\R^n} \frac{t^{2\sigma}}{(|x-y|^2+t^2)^{\frac{n+2\sigma}{2}}}\eta(|y|)v(y)\,\ud y\\
&\quad+  \beta(n,\sigma) \int_{\R^n} \frac{t^{2\sigma}}{(|x-y|^2+t^2)^{\frac{n+2\sigma}{2}}}(1-\eta(|y|))v(y)\,\ud y\\
&\quad=:V_1+V_2.
\end{align*}
We have for all $k=1,2,\cdots,2m$ that 
\begin{align*}
\nabla^k_x V_1(x,t)&=\beta(n,\sigma) \int_{\R^n}\eta(|y|)v(y)\nabla_x^k \frac{t^{2\sigma}}{(|x-y|^2+t^2)^{\frac{n+2\sigma}{2}}}\,\ud y\\
&=\beta(n,\sigma) \int_{\R^n} \frac{t^{2\sigma}}{(|x-y|^2+t^2)^{\frac{n+2\sigma}{2}}}\nabla_y^k (\eta(|y|)v(y))\,\ud y.
\end{align*}
Denote $\widetilde \Delta_t=\partial_{tt}+\frac{b}{t}\partial_t$. Since
\begin{align*}
 V_1(x,t)&= \beta(n,\sigma) \int_{\R^n} \frac{t^{2\sigma}(\eta(|y|)v(y)-T_{k,x}(y))}{(|x-y|^2+t^2)^{\frac{n+2\sigma}{2}}}\,\ud y + \beta(n,\sigma) \int_{\R^n} \frac{t^{2\sigma}T_{k,x}(y)}{(|x-y|^2+t^2)^{\frac{n+2\sigma}{2}}}\,\ud y\\
 &= \beta(n,\sigma) \int_{\R^n} \frac{t^{2\sigma}(\eta(|y|)v(y)-T_{k,x}(y))}{(|x-y|^2+t^2)^{\frac{n+2\sigma}{2}}}\,\ud y+ \sum_{j=0}^{2k} c_j t^j,
\end{align*}
where $T_{k,x}(y)$ is the Taylor expansion of $v$ at $x$ up to order $2k$.
Then
\begin{align*}
\widetilde \Delta_t^k V_1(x,t)&= \beta(n,\sigma) \int_{\R^n}(v(y)-T_{k,x}(y))\Delta_t^k \frac{t^{2\sigma}}{(|x-y|^2+t^2)^{\frac{n+2\sigma}{2}}}\,\ud y+\tilde c_{2k}
\end{align*}
and
\begin{align*}
&\int_{\R^n}(\eta(|y|)v(y)-T_{k,x}(y))\Delta_t^k \frac{t^{2\sigma}}{(|x-y|^2+t^2)^{\frac{n+2\sigma}{2}}}\,\ud y\\
& \le C \int_{B_\va(x)}|y-x|^{2k+\delta} \frac{t^{2(\sigma-k)}}{(|x-y|^2+t^2)^{\frac{n+2\sigma}{2}}}\,\ud y\\
&\quad+ C t^{2(\sigma-k)}\int_{\R^n\setminus B_\va(x)}\frac{|\eta(|y|)v(y)-T_{k,x}(y)|}{|x-y|^{n+2\sigma}}\,\ud y,
\end{align*}
where $\va$ is some small number. Also, we observe that
\begin{align*}
|\nabla^k_\ell \nabla^k_x V_2(x,t)|\le Ct^{2\sigma-\ell},
\end{align*}
where we used that \eqref{eq:cut-offassumption}.

This shows that $\Delta^k_x V(x,t)$ and $\widetilde \Delta_t^k V(x,t)$ are both H\"older continuous on $\overline{\mathcal{B}_{r_3}^+\setminus \mathcal B_{r_2}^+} $. Using the assumption \eqref{eq:cut-offassumption} and similar arguments, one can show show that $\Delta_b^k V(x,t)$ is H\"older  continuous on $\overline{\mathcal{B}_{r_3}^+\setminus \mathcal B_{r_2}^+} $ for all $k=0,1,2,\cdots,m$, and so is $t^b\pa_\nu \Delta_b^m V $. 

Let $\eta_R(s)=\eta(s/R)$, $R>0$. Write
\begin{align*}
 V(x,t)&= \beta(n,\sigma) \int_{\R^n} \frac{t^{2\sigma}}{(|x-y|^2+t^2)^{\frac{n+2\sigma}{2}}}\eta_R(|y|)v(y)\,\ud y\\
&\quad+  \beta(n,\sigma) \int_{\R^n} \frac{t^{2\sigma}}{(|x-y|^2+t^2)^{\frac{n+2\sigma}{2}}}(1-\eta_R(|y|))v(y)\,\ud y\\
&\quad=:V_{1,R}+V_{2,R}.
\end{align*}
Since
\[
(-1)^{m+1}\lim_{t\to 0} t^{b} \pa_t \Delta_b^m V_{1,R}(x,t)= N_{n,\sigma} (-\Delta)^\sigma (\eta_R(|x|)v(x)),
\]
to verify \eqref{eq:extensionneumann} for $V$ in $B_1$, we only need to show that
\begin{align}
\lim_{R\to\infty}(-1)^{m+1}\lim_{t\to 0} t^{b} \pa_t \Delta_b^m V_{2,R}(x,t)&= 0,\label{eq:integraltailzero}\\
\lim_{R\to\infty} (-\Delta)^\sigma \Big((1-\eta_R(|x|))v(x)\Big)&=0.\label{eq:tailzeroharmonic}
\end{align}
The identity \eqref{eq:integraltailzero} follows from the observation that
\[
\lim_{t\to 0} |t^{b} \pa_t \Delta_b^m V_{2,R}(x,t)|\le C\int_{\{|y|\ge R\}} \frac{v(y)}{|y|^{n+2\sigma}}\,\ud y\to 0\quad\mbox{as }R\to\infty.
\]
The identity \eqref{eq:tailzeroharmonic} follows from the estimate that
\begin{align*}
 &(-\Delta)^\sigma \Big((1-\eta_R(|x|))v(x)\Big)\\
 &= (-\Delta)^{\sigma-m} (-\Delta)^m \Big((1-\eta_R(|x|))v(x)\Big)\\
 &=C\int_{\{|y|>R\}}\frac{|(-\Delta)^mu(y)|}{|x-y|^{n+2(\sigma-m)}}\,\ud y+C\sum_{k=0}^{2m-1}\int_{\{R<|y|<2R\}}\frac{|\nabla^{2m-k}(1-\eta_R)(y)\nabla^k v(y)|}{|x-y|^{n+2(\sigma-m)}}\,\ud y\\
 &\le C\int_{\{|y|>R\}}\frac{|(-\Delta)^mu(y)|}{1+|y|^{n+2(\sigma-m)}}\,\ud y+C\sum_{k=0}^{2m-1}\int_{\{R<|y|<2R\}}\frac{|\nabla^k v(y)|}{R^{2m-k}|y|^{n+2(\sigma-m)}}\,\ud y\\
 &\le C\sum_{k=0}^{2m}\int_{\{|y|>R\}}\frac{|\nabla^k u(y)|}{1+|y|^{n+2\sigma-k}}\,\ud y\to 0\quad\mbox{as }R\to\infty.
\end{align*}
This verifies \eqref{eq:extensionneumann} for $V$. The verification of \eqref{eq:extensionboundaryzero} for $V$ can be done similarly.
\end{proof}

Our solution $v$ defined in \eqref{eq:definitionofv} satisfies \eqref{eq:cut-offassumption} since $h$ is assumed to satisfy \eqref{eq:assumptionsh}. However, our $v$ may have a singularity at the origin. Nevertheless, from Lemma \ref{lem:estimates}, $|\nabla^k u(x)| =O(|x|^{-\frac{n-2\sigma}{2}-k})$ for $x \in B_1\setminus \{0\}$. This is actually enough to ensure that \eqref{eq:extensionneumann} and \eqref{eq:extensionboundaryzero} hold, using another cut-off as follows.

\begin{prop}\label{lem:cut-off2}
Let $v\in W^{2m,1}_{loc}(\R^n\setminus\{0\}) \cap C^{\infty}(B_2\setminus\{0\})$ be a positive function in $\R^n\setminus\{0\}$ satisfying \eqref{eq:cut-offassumption}, where $m=[\sigma]$. Suppose in addition that 
\begin{align}\label{eq:cut-offassumption0}
\int_{B_1}\frac{|\nabla^kv(y)|}{|y|^{2m-k}}\,\ud y<\infty, \quad k=0,1,\cdots, 2m. 
\end{align}
Let $V$ be as in \eqref{eq:poissonextension}. Then the left hand sides of \eqref{eq:extensionboundaryzero} and \eqref{eq:extensionneumann} are H\"older continuous functions in $\overline{\B_1^+\setminus \B^+_{1/2}}$, and both \eqref{eq:extensionboundaryzero} and \eqref{eq:extensionneumann} hold for all $x\in B_1\setminus B_{1/2}$.
\end{prop}
\begin{proof}
The proof for that the left hand sides of \eqref{eq:extensionboundaryzero} and \eqref{eq:extensionneumann} are H\"older continuous functions in $\overline{\B_1^+\setminus \B^+_{1/2}}$ is almost identical to that in Proposition \ref{lem:cut-off1}, replacing the cut-off $\eta$ by the new cut-off function $\tilde\eta$ which is equal to $1$ in $[1/4,5/4]$ and is identically zero outside of $[1/8,3/2]$.

To prove that both \eqref{eq:extensionboundaryzero} and \eqref{eq:extensionneumann} hold for all $x\in B_1\setminus B_{1/2}$, we use the cut-off function $\eta_\va(s)=\eta(s/\va)$, where $\va>0$ small, and $\eta$ is the one in the proof of  Proposition \ref{lem:cut-off1}. To verify \eqref{eq:extensionneumann} for $V$ in $B_1\setminus B_{1/2}$, using Proposition \ref{lem:cut-off1}, we only need to show that
\begin{align}
\lim_{\va\to 0}(-1)^{m+1}\lim_{t\to 0} t^{b} \pa_t \Delta_b^m \widetilde V_{\va}(x,t)&= 0,\label{eq:integraltailzero2}\\
\lim_{\va\to0} (-\Delta)^\sigma \Big(\eta_\va(|x|)v(x)\Big)&=0.\label{eq:tailzeroharmonic2}
\end{align}
where
\[
\widetilde V_\va=\beta(n,\sigma) \int_{\R^n} \frac{t^{2\sigma}}{(|x-y|^2+t^2)^{\frac{n+2\sigma}{2}}}\eta_\va(|y|)v(y)\,\ud y.
\]

For \eqref{eq:integraltailzero2}, we observe that 
\[
|(-1)^{m+1}\lim_{t\to 0} t^{b} \pa_t \Delta_b^m \widetilde V_{\va}(x,t)|\le C \int_{\{|y|\le \va\}} v(y)\,\ud y\to 0\quad\mbox{as }\va\to 0^+.
\]
For \eqref{eq:tailzeroharmonic2}, similar to the proof of \eqref{eq:tailzeroharmonic}, we have
\begin{align*}
(-\Delta)^\sigma \Big(\eta_\va(|x|)v(x)\Big)&\le \int_{|y|\le 2\va} |(-\Delta)^m v(y)|\,\ud y + C\sum_{k=0}^{2m-1} \int_{\va/8\le  |y|\le 2\va}\va^{k-2m} |\nabla^k v(y)|\,\ud y \\
&\le C\sum_{k=0}^{2m} \int_{|y|\le 2\va} |y|^{k-2m} |\nabla^k v(y)|\,\ud y\to 0\quad\mbox{as }\va\to 0^+.
\end{align*}

This finishes the proof of this lemma.
\end{proof}

\section{Some technical estimates for the non-integer cases}\label{appendix:technical}

This appendix is for the proof of the lower bound in Proposition \ref{prop:zero} when $\sigma$ is not an integer. The idea is similar to the case that $\sigma$ is an integer. But we need to work in the framework of the extension \eqref{eq:poissonextension}, and therefore, we need to establish all the needed ingredients for the extension $V$. For $\sigma\in (0,1)$, this analysis has been carried out in \cite{CJSX}. So we will focus on the case $\sigma>1$.

 Recall $v$ is defined in \eqref{eq:definitionofv} and  is a solution of
\be \label{eq:v}
(-\Delta)^\sigma v= c(n,\sigma) v^{\frac{n+2\sigma}{n-2\sigma}} \quad \mbox{in }B_2\setminus\{0\}, \quad v>0\quad\mbox{in }\R^n.
\ee
Define
\begin{equation}\label{eq:pohozaevnoninteger}
\begin{split}
P(v,r)&=\frac{n-2\sigma}{2} \int_{\pa'' \B_r^+} t^b\pa_\nu \Delta_b^m V V +r\int_{\pa'' \B_r^+} t^b\pa_\nu \Delta_b^m V \partial_\nu V \\
&\quad+\int_{\pa'' \B_r^+} Q_m(V) + \frac{(-1)^m(n-2\sigma)N_\sigma c(n,\sigma)}{2n} r\int_{\pa B_r}  v^{\frac{2n}{n-2\sigma}}.
\end{split}
\end{equation}
where $V$ is the Poisson extension of $v$ defined in \eqref{eq:poissonextension}. Then 
\begin{prop}\label{prop:pohozaev22}
Let $u$ and $h$ be as in Proposition \ref{prop:zero}, and $v$ be as in \eqref{eq:definitionofv}, and $P(v,r)$ be as in \eqref{eq:pohozaevnoninteger}. Then $P(v,r)$ is independent of $r$.
\end{prop}

\begin{proof}   
We apply Proposition  \ref{prop:GPYsigma} with $E=(\mathcal B_R^+\setminus \mathcal B_{r}^+)\cap \{(x,t):\ t=\va\}$ and $U=V$. By Proposition \ref{lem:cut-off2}, all the integrands in $Q_m$ are continuous functions, and moreover, 
\[
\lim_{\va\to 0^+} \int_{(\mathcal B_R^+\setminus \mathcal B_{r}^+)\cap \{(x,t):\ t=\va\}} Q_m(V)=0,\quad\forall\ 2>R> r>0.
\]
The conclusion follows by using the equations \eqref{eq:v}, \eqref{eq:extensionharmonic} and \eqref{eq:extensionneumann},  integration by parts, and sending $\va\to 0^+$ in the end.
\end{proof} 
 
We will derive a Harnack inequality for $V$. To achieve this, we first need a Harnack inequality for the tail of the Poisson convolution. 

\begin{lem}\label{lem:poissonharnacktail}
Let $u$ and $h$ be as in Proposition \ref{prop:zero} and  $v$ be as in \eqref{eq:definitionofv}. Then for any $0<r<1/4$ and $3/4\le |x|\le 5/4$, we have 
\begin{equation}\label{eq:harnackpossiontail}
\int_{\{|y|<1/2\} \cup \{|y|>3/2\}} |x-y|^{-(n+2\sigma)} v(ry) \,\ud y \le C v(r x), 
\end{equation}
where $C>0$ is independent of $r$. 
\end{lem}
\begin{proof}
 since
\begin{align*}
w_r(x)=\int_{B_{r/2}}\frac{w_r(z)^{\frac{n+2\sigma}{n-2\sigma}}}{|x-z|^{n-2\sigma}}\,\ud z+ h_r(x),
\end{align*}
where $w_r(x)=r^{\frac{n-2\sigma}{2}}v(rx)$ and $h_r(x)=r^{\frac{n-2\sigma}{2}}h(rx)$, we have
\begin{align*}
&\int_{\{|y|<1/2\} \cup \{|y|>3/2\}} |x-y|^{-(n+2\sigma)} w_r(y) \,\ud y\\
&=\int_{\{|y|<1/2\} \cup \{|y|>3/2\}} |x-y|^{-(n+2\sigma)} \int_{B_{r/2}}\frac{w_r(z)^{\frac{n+2\sigma}{n-2\sigma}}}{|y-z|^{n-2\sigma}}\,\ud z \,\ud y\\
&\quad+\int_{\{|y|<1/2\} \cup \{|y|>3/2\}} |x-y|^{-(n+2\sigma)} h_r(y) \,\ud y.
\end{align*}
Note that
\begin{align*}
\int_{\{|y|<1/2\} \cup \{|y|>3/2\}} |x-y|^{-(n+2\sigma)} \int_{|z|\le 1/2}\frac{w_r(z)^{\frac{n+2\sigma}{n-2\sigma}}}{|y-z|^{n-2\sigma}}\,\ud z \,\ud y&\le C \int_{|z|\le 1/2}w_r(z)^{\frac{n+2\sigma}{n-2\sigma}}\,\ud z\\
&\le C\int_{|z|\le 1/2}\frac{w_r(z)^{\frac{n+2\sigma}{n-2\sigma}}}{|x-z|^{n-2\sigma}}\,\ud z\\
&\le Cw_r(x),
\end{align*}
where we used Fubini's theorem and the positivity of $h$. And
\begin{align*}
&\int_{\{|y|<1/2\} \cup \{|y|>3/2\}} |x-y|^{-(n+2\sigma)} \int_{1/2\le |z|\le 3/2}\frac{w_r(z)^{\frac{n+2\sigma}{n-2\sigma}}}{|y-z|^{n-2\sigma}}\,\ud z \,\ud y\\
&\quad\le C w_r(x)\int_{\{|y|<1/2\} \cup \{|y|>3/2\}} |x-y|^{-(n+2\sigma)}\int_{1/2\le |z|\le 3/2}\frac{1}{|y-z|^{n-2\sigma}}\,\ud z \,\ud y\\
&\quad\le C w_r(x),
\end{align*}
where we used Proposition \ref{prop:upbound} and Lemma \ref{lem:harnack} in the second inequality. Finally,
\begin{align*}
&\int_{\{|y|<1/2\}} |x-y|^{-(n+2\sigma)} \int_{3/2\le |z|\le r/2}\frac{w_r(z)^{\frac{n+2\sigma}{n-2\sigma}}}{|y-z|^{n-2\sigma}}\,\ud z \,\ud y\\
&\quad\le C \int_{\{|y|<1/2\}}\int_{3/2\le |z|\le r/2}\frac{w_r(z)^{\frac{n+2\sigma}{n-2\sigma}}}{|x-z|^{n-2\sigma}}\,\ud z \,\ud y\\
&\quad\le Cw_r(x),
\end{align*}
and
\begin{align*}
&\int_{\{|y|>3/2\}} |x-y|^{-(n+2\sigma)} \int_{3/2\le |z|\le r/2}\frac{w_r(z)^{\frac{n+2\sigma}{n-2\sigma}}}{|y-z|^{n-2\sigma}}\,\ud z \,\ud y\\
&\quad\le \int_{3/2\le |z|\le r/2} w_r(z)^{\frac{n+2\sigma}{n-2\sigma}}\int_{\{|y|>3/2\}} |y|^{-(n+2\sigma)} \frac{1}{|y-z|^{n-2\sigma}}\,\ud y\,\ud z\\
&\quad\le C \int_{3/2\le |z|\le r/2}\frac{w_r(z)^{\frac{n+2\sigma}{n-2\sigma}}}{|z|^{n-2\sigma}}\,\ud z\\
&\quad\le C \int_{3/2\le |z|\le r/2}\frac{w_r(z)^{\frac{n+2\sigma}{n-2\sigma}}}{|x-z|^{n-2\sigma}}\,\ud z\\
&\quad\le Cw_r(x).
\end{align*}
Therefore,
\[
\int_{\{|y|<1/2\} \cup \{|y|>3/2\}} |x-y|^{-(n+2\sigma)} \int_{B_{r/2}}\frac{w_r(z)^{\frac{n+2\sigma}{n-2\sigma}}}{|y-z|^{n-2\sigma}}\,\ud z \,\ud y\le Cw_r(x).
\]

On the other hand, we have
\begin{align*}
\int_{\{|y|<1/2\} \cup \{3/2<|y|\le 3/(2r)\}} |x-y|^{-(n+2\sigma)} h_r(y) \,\ud y
&\le Cr^{\frac{n-2\sigma}{2}}\max_{B_{3/2}} h\\
&\le Cr^{\frac{n-2\sigma}{2}}h(rx)\\
&\le Cw_r(x),
\end{align*}
and
\begin{align*}
\int_{\{|y|> 3/(2r)\}} |x-y|^{-(n+2\sigma)} h_r(y) \,\ud y&\le Cr^{\frac{n+2\sigma}{2}}\int_{\{|y|> 3/2\}} |rx-y|^{-(n+2\sigma)} h(y) \,\ud y\\
&\le Cr^{\frac{n+2\sigma}{2}} \quad\mbox{by the assumption \eqref{eq:assumptionsh}}\\
&\le Cr^{\frac{n+2\sigma}{2}} h(rx)\le  Cw_r(x).
\end{align*}
Hence
\[
\int_{\{|y|<1/2\} \cup \{|y|>3/2\}} |x-y|^{-(n+2\sigma)} h_r(y) \,\ud y\le  Cw_r(x).
\]
Therefore,
\begin{equation}\label{eq:harnackw}
\int_{\{|y|<1/2\} \cup \{|y|>3/2\}} |x-y|^{-(n+2\sigma)} w_r(y) \,\ud y\le  Cw_r(x).
\end{equation}
This finishes proving \eqref{eq:harnackpossiontail}.
\end{proof}

\begin{lem}\label{lem:poissonharnackV}
Let $u$ and $h$ be as in Proposition \ref{prop:zero},  $v$ be as in \eqref{eq:definitionofv} and $V$ be as in \eqref{eq:poissonextension}. Then for any $0<r<1/4$,
\[
\sup_{ \B_{5r/4}^+\setminus  \B_{3r/4}^+ } V \le C \inf_{ \B_{5r/4}^+\setminus  \B_{3r/4}^+ } V, 
\]
where $C>0$ is independent of $r$. 
\end{lem}
\begin{proof}
We shall use some arguments from \cite{X}.    

Let $W_r(X)=r^{\frac{n-2\sigma}{2}}V(rX)$ and $w_r(X)=r^{\frac{n-2\sigma}{2}}v(rx)$. Let $\eta$ be a nonnegative smooth cut-off function such that $\eta\equiv 1$ in $[1/2,3/2]$ and is identically zero outside of $[1/4,7/4]$. By the definition of $V$, we have 
\begin{align*}
W_r(X)&= \beta(n,\sigma) \int_{\R^n} \frac{t^{2\sigma}}{(|x-y|^2+t^2)^{\frac{n+2\sigma}{2}}} \eta(|y|)w_r(y)\,\ud y \\& \quad +  \beta(n,\sigma) \int_{\{ |y|<1/2\} \cup \{ |y|>3/2\} } \frac{t^{2\sigma}}{(|x-y|^2+t^2)^{\frac{n+2\sigma}{2}}} (1-\eta(|y|))w_r(y)\,\ud y \\&
=:\Psi_1(X)+\Psi_2(X).
\end{align*}
Using Lemma \ref{lem:harnack}, we have 
\[
\sup_{ \B_{5/4}^+\setminus  \B_{3/4}^+ } \Psi_1 \le \sup_{B_{7/4}\setminus B_{1/4}}w_r\le C\inf_{B_{7/4}\setminus B_{1/4}}w_r\le C \inf_{ \B_{5/4}^+\setminus  \B_{3/4}^+ } \Psi_1.
\]
Direction calculations yield that
\[
\sup_{ \B_{5/4}^+\setminus  \B_{3/4}^+ } \frac{\Psi_2}{t^{2\sigma}} \le C \inf_{ \B_{5/4}^+\setminus  \B_{3/4}^+ } \frac{\Psi_2}{t^{2\sigma}}.
\] 
By Step 1, we have $ \inf_{ \B_{5/4}^+\setminus  \B_{3/4}^+ } \frac{\Psi_2}{t^{2\sigma}} \le w_r(x)$ for any $|x|=1$. Using Lemma \ref{lem:harnack} again, we have
\[
\sup_{ \B_{5/4}^+\setminus  \B_{3/4}^+ } \Psi_2\le C\inf_{ B_{3/2}\setminus  B_{1/2} } w_r\le C \inf_{ \B_{5/4}^+\setminus  \B_{3/4}^+ } W_r.
\]
Hence, we conclude that 
\[
\sup_{ \B_{5/4}^+\setminus  \B_{3/4}^+ } W_r \le C \inf_{ \B_{5/4}^+\setminus  \B_{3/4}^+ } W_r.
\] 
Scaling back, we proved the Harnack inequality of $V$.  
\end{proof}

Moreover, using \eqref{eq:harnackw} and direct calculations, we have for $(x,t)\in \B_{5/4}^+\setminus  \B_{3/4}^+ $,
\begin{align*}
|\nabla_t^\ell  \nabla_x^k\Psi_2(x,t)|&\le C t^{2\sigma-\ell}\quad\mbox{for all }k\ge 1, 1\le \ell<2\sigma.
\end{align*}
Following the above proof for that the left hand sides of \eqref{eq:extensionboundaryzero} and \eqref{eq:extensionneumann} are continuous on $\overline{\mathcal{B}_{R}^+\setminus \mathcal B_r^+} $, one can show that
\[
|\nabla^{k}_{x,t}\Psi_1(x,t)|+|t^b\partial_t \Delta_b^m \Psi_1(x,t)|\le C,\quad k=0,\cdots, 2m.
\]
Hence, for $(x,t)\in \B_{5/4}^+\setminus  \B_{3/4}^+ $,
\[
|\nabla^{k}_{x,t}W_r(x,t)|+|t^b\partial_t \Delta_b^m W_r(x,t)|\le C,\quad k=0,\cdots, 2m.
\]

\begin{lem} \label{lem:estimatepohozaev}
Let $u$ and $h$ be as in Proposition \ref{prop:zero},  $v$ be as in \eqref{eq:definitionofv} and 
\begin{equation}\label{eq:extendvarphi}
\varphi_i(y)=\frac{v(r_i y)}{u(r_i e_n)}.
\end{equation}
 Let $V= \mathcal{P}*v$ and $\Phi_i= \mathcal{P}*\varphi_i$  Then
\begin{align}
\lim_{i\to\infty} P(v,r_i)&=0,\label{eq:appendixtends0}\\
\Phi_i&\to \frac{1}{2|X|^{n-2\sigma}}+\frac 12, \label{eq:appendixfundamental}\\
\lim_{i\to\infty} \int_{\pa'' \B_1^+} \widetilde Q_m(\Phi_i) &= \int_{\pa'' \B_1^+} \widetilde Q_m\left(\frac{1}{2|X|^{n-2\sigma}}+\frac 12\right), \label{eq:appendixlimit}
\end{align}
where
\begin{equation}\label{eq:tildeQ}
\int_{\pa'' \B_r^+} \widetilde Q_m(U) =\frac{n-2\sigma}{2} \int_{\pa'' \B_r^+} t^b\pa_\nu \Delta_b^m U U +r\int_{\pa'' \B_r^+} t^b\pa_\nu \Delta_b^m U \partial_\nu U+\int_{\pa'' \B_r^+} Q_m(U). 
\end{equation}
\end{lem}
\begin{proof} 
We will do the case for $\sigma>1$, since the above analysis for $\sigma\in (0,1)$ has been done in \cite{CJSX}.

 Choosing $r=r_i$, then we have
\[
\sup_{ \B_{5/4}^+\setminus  \B_{3/4}^+ } \Phi_i \le C \inf_{ \B_{5/4}^+\setminus  \B_{3/4}^+ } \Phi_i.
\]
and $\Phi_i$ is locally uniformly bounded, and $(x,t)\in \B_{5/4}^+\setminus  \B_{3/4}^+ $ that
\begin{equation}\label{eq:localestimateWr}
|D^{k}_{x,t}\Phi_i(x,t)|+|t^b\partial_t \Delta_b^m \Phi_i(x,t)|\le C,\quad k=0,\cdots, 2m.
\end{equation}
Hence for all $|(x,t)|=r_i, k=1,2,\cdots,2m$,
\begin{align*}
|D^k_{x,t} V(x,t)|&\le C r_i^{-k} u(r_ie_n)=o(1)r_i^{-\frac{n-2\sigma}{2}-k}\\
|t^b\partial_t \Delta_b^m V(x,t)|&\le C r_i^{-2m-1+b} u(r_ie_n)=o(1)r_i^{-\frac{n-2\sigma}{2}-2m-1+b}.
\end{align*}
This estimate will imply \eqref{eq:appendixtends0}.

We arbitrarily fix $0<r_1<r_2$, and we prove \eqref{eq:appendixfundamental} for $r_1\le |x|\le r_2$, $t>0$. From the proof of \eqref{eq:converg2}, we see that $\varphi_j(x)\le C|x|^{2\sigma-n}$ for $|x|\le 1/2$, where $C$ is independent of $j$. We estimate as follows:
\begin{align*}
\Phi_i&=\beta(n,\sigma) \int_{\R^n} \frac{t^{2\sigma}}{(|x-y|^2+t^2)^{\frac{n+2\sigma}{2}}}\varphi_i(y)\,\ud y\\
&=\beta(n,\sigma) \int_{\va\le |y|\le R} \frac{t^{2\sigma}}{(|x-y|^2+t^2)^{\frac{n+2\sigma}{2}}}\varphi_i(y)\,\ud y\\
&\quad+\beta(n,\sigma) \int_{\{|y|<\va\}\cup\{|y|>R\}} \frac{t^{2\sigma}}{(|x-y|^2+t^2)^{\frac{n+2\sigma}{2}}}\varphi_i(y)\,\ud y\\
&\quad=:I_i(x,t)+II_i(x,t),
\end{align*}
where $\va<r_1$ is small and $R>r_2$ is large. It follows from  \eqref{eq:converg2}  that
\[
I_i(x,t)\to \beta(n,\sigma) \int_{\va\le |y|\le R} \frac{t^{2\sigma}}{(|x-y|^2+t^2)^{\frac{n+2\sigma}{2}}}\left(\frac{1}{2|y|^{n-2\sigma}}+\frac 12\right)\,\ud y.
\]
Since $II_i$ is bounded in $C^1(\B^+_{r_2}\setminus\B^+_{r_1})$, we know that $II_i(x)\to f_{\va,R}(x)$ for some function $f$. Then as in the proof of \eqref{eq:converg2}, one can show that $\lim_{\va\to 0, R\to \infty}f(x)$ is a constant function. Therefore,
\begin{align*}
\Phi_i&\to  \beta(n,\sigma) \int_{\R^n} \frac{t^{2\sigma}}{(|x-y|^2+t^2)^{\frac{n+2\sigma}{2}}}\left(\frac{1}{2|y|^{n-2\sigma}}+\frac 12\right)\,\ud y+c_0\\
&=\frac{1}{2|X|^{n-2\sigma}}+\frac 12  +c_0,
\end{align*}
for some constant $c_0$. Here, we used the following identity
\begin{equation}\label{eq:fundamentaltobubble}
\beta(n,\sigma) \int_{\R^n} \frac{t^{2\sigma}}{(|x-y|^2+t^2)^{\frac{n+2\sigma}{2}}}\frac{1}{|y|^{n-2\sigma}}\,\ud y=\frac{1}{|X|^{n-2\sigma}}
\end{equation}
of which the proof is postponed to the end. Since $\Phi_i$ is uniformly bounded in $C^1 (\overline{\mathcal{B}_2^+}\setminus \mathcal{B}_{1/2}^+)$, and $\varphi_i(x)\to \frac{1}{2|x|^{n-2\sigma}}+\frac 12$, we have that $c_0=0$ and
\[
\Phi_i\to \frac{1}{2|X|^{n-2\sigma}}+\frac 12 \quad\mbox{in  } C^1 (\overline{\mathcal{B}_2^+}\setminus \mathcal{B}_{1/2}^+). 
\]
By \eqref{eq:localestimateWr}, one can show that $\Phi_i\to 1/(2|X|^{n-2\sigma})+1/2$ in the needed norms to guarantee that \eqref{eq:appendixlimit} holds.

Now we are left to prove \eqref{eq:fundamentaltobubble}. We have
\begin{align*}
&\beta(n,\sigma) \int_{\R^n} \frac{t^{2\sigma}}{(|x-y|^2+t^2)^{\frac{n+2\sigma}{2}}}\frac{1}{|y|^{n-2\sigma}}\,\ud y\\
&=t^{2\sigma-n}\beta(n,\sigma) \int_{\R^n} \frac{1}{(|\frac{x}{t}-y|^2+1)^{\frac{n+2\sigma}{2}}}\frac{1}{|y|^{n-2\sigma}}\,\ud y\\
&= t^{2\sigma-n}\beta(n,\sigma) \int_{\R^n} \left(\frac{1}{(|y|^2+1)^{\frac{n-2\sigma}{2}}}\right)^{\frac{n+2\sigma}{n-2\sigma}}\frac{1}{|\frac{x}{t}-y|^{n-2\sigma}}\,\ud y\\
&= t^{2\sigma-n}\frac{1}{(|\frac{x}{t}|^2+1)^{\frac{n-2\sigma}{2}}}\\
&=\frac{1}{|X|^{n-2\sigma}},
\end{align*}
where in the third equality we used the fact that $\bar u_1$ defined in \eqref{eq:bubblelambda} is a solution of 
\[
u(x)=\beta(n,\sigma)\int_{\R^n}\frac{u(y)^{\frac{n+2\sigma}{n-2\sigma}}}{|x-y|^{n-2\sigma}}\,\ud y.
\]
\end{proof}

\end{appendices}

\bigskip
\bigskip

\noindent T. Jin

\noindent Department of Mathematics, The Hong Kong University of Science and Technology\\
Clear Water Bay, Kowloon, Hong Kong\\[1mm]
Email: \textsf{tianlingjin@ust.hk}

\medskip

\noindent J. Xiong

\noindent School of Mathematical Sciences, Beijing Normal University\\
Beijing 100875, China\\[1mm]
Email: \textsf{jx@bnu.edu.cn}

\end{document}